\documentclass[11pt]{amsart}
\usepackage{amsfonts,amscd,latexsym,amsmath,amssymb,enumerate,verbatim,amsthm,epsfig,alltt,color}
\usepackage[shortlabels]{enumitem}

\newcommand{\Homeo}{\mathrm{Homeo}}
\newcommand{\Act}{\mathrm{Act}}

\newcommand{\Rat}{\mathbb{Q}}

\newcommand{\Nat}{\mathbb{N}}
\newcommand{\Int}{\mathbb{Z}}

\newcommand{\FF}{\mathcal{F}}
\newcommand{\GG}{\mathcal{G}}

\newcommand{\AAA}{\mathcal{A}}

\newcommand{\NN}{\mathcal{N}}
\newcommand{\UU}{\mathcal{U}}
\newcommand{\SH}{\mathcal{S}}
\newcommand{\CC}{\mathcal{C}}

\newcommand{\QQ}{\mathcal{Q}}
\newcommand{\PP}{\mathcal{P}}
\newcommand{\RR}{\mathcal{R}}

\newcommand{\VV}{\mathcal{V}}
\newcommand{\SSS}{\mathcal{S}}

\theoremstyle{plain}

\newtheorem{theorem}{\bf Theorem}[section]
\newtheorem{lemma}[theorem]{\bf Lemma}
\newtheorem{proposition}[theorem]{\bf Proposition}
\newtheorem{corollary}[theorem]{\bf Corollary}
\newtheorem{fact}[theorem]{\bf Fact}

\newtheorem{thmx}{Theorem}

\theoremstyle{definition}
\newtheorem{definition}[theorem]{\bf Definition}
\newtheorem{example}[theorem]{\bf Example}
\newtheorem{remark}[theorem]{\bf Remark}

\newtheorem{problem}[theorem]{\bf Problem}
\newtheorem{question}[theorem]{\bf Question}
\newtheorem{construction}[theorem]{\bf Construction}

\begin{document}
	\title[Strong topological Rokhlin property]{Strong topological Rokhlin property, shadowing, and symbolic dynamics of countable groups}
	\author{Michal Doucha}
	\address{Institute of Mathematics\\
		Czech Academy of Sciences\\
		\v Zitn\'a 25\\
		115 67 Praha 1\\
		Czechia}
	\email{doucha@math.cas.cz}
	\urladdr{https://users.math.cas.cz/~doucha/}
	\keywords{countable groups, Cantor dynamics, symbolic dynamics, subshifts of finite type, generic actions, shadowing}
	\thanks{The author was supported by by the GA\v{C}R project 22-07833K and RVO: 67985840.}
	
\begin{abstract}
A countable group $G$ has the strong topological Rokhlin property (STRP) if it admits a continuous action on the Cantor space with a comeager conjugacy class. We show that having the STRP is a symbolic dynamical property. We prove that a countable group $G$ has the STRP if and only if certain sofic subshifts over $G$ are dense in the space of subshifts. A sufficient condition is that isolated shifts over $G$ are dense in the space of all subshifts. 
		
We provide numerous applications including the proof that a group that decomposes as a free product of finite or cyclic groups has the STRP. We show that finitely generated nilpotent groups do not have the STRP unless they are virtually cyclic; the same is true for many groups of the form $G_1\times G_2\times G_3$ where each factor is recursively presented. We show that a large class of non-finitely generated groups do not have the STRP, this includes any group with infinitely generated center and the Hall universal locally finite group.

We find a very strong connection between the STRP and shadowing, a.k.a. pseudo-orbit tracing property. We show that shadowing is generic for actions of a finitely generated group $G$ if and only if $G$ has the STRP.
\end{abstract}

\dedicatory{To the memory of my father Franti\v sek Doucha (1952-2022)}
	
\maketitle

\setcounter{tocdepth}{1}
\tableofcontents

\section{Introduction}
This paper presents and explores certain connections between generic group actions on the Cantor space and the structure of subshifts of finite type and sofic subshifts over these groups.

Let us start with some history and motivation. One of the earliest occurencies of `genericity results' in measurable and topological dynamics was Halmos' paper \cite{Halm}, following the previous result of Oxtoby and Ulam in \cite{OxUl}, where he showed that ergodicity and weak mixing are generic properties among p.m.p. bijections of the standard probability space, removing the fears that such useful properties might actually be rare among general p.m.p. bijections (see \cite{AlpPra02} for the connections betwen Halmos's and Oxtoby-Ulam's papers). He also showed that the conjugacy class of any aperiodic p.m.p. bijection is dense, a result which is also attributed to Rokhlin and often identified with the Rokhlin lemma. This also explains why the name `Rokhlin' is attached to results of this kind, we refer to the survey \cite{GlaWei08}. We refer to monographs \cite{AlpPra} and \cite{AkHurKen} for more historical background and many results of this sort and to \cite{Hoch08} for some more recent developments in the topological case.

It is our aim here to investigate these problems, but not only for single invertible transformations, or in the language that we shall use, for actions of the group of integers, but for general countable group actions. This connects this area of research with combinatorial and geometric group theory and reveals sometimes surprising and beautiful differences between dynamical properties of geometrically different groups. We continue with more examples.

In topological dynamical category, Glasner and Weiss showed in \cite{GlaWei01} that there is an action of the integers on the Cantor space with a dense conjugacy class. This has been extended to actions of $\Int^d$ by Hochman \cite{Hoch12} and his proof works for any countable group, so density of conjugacy classes is not an interesting phenomenon in this case although interesting differences occur if we require a dense conjugacy class that is computable, see again \cite{Hoch12}. Glasner, Thouvenot, and Weiss in \cite{GlaThouWei} also showed that for any countable group $G$, a generic p.m.p. action has a dense conjugacy class (this is also an unpublished and indepedently proved result of Hjorth).

It is then of particular interest whether in a given setting or given category there is an action which itself is generic, meaning its conjugacy class is comeager, thereby reducing the investigation of generic properties in that category to the properties of that particular action. This is known to be false for the integer actions on the standard probability space by del Junco \cite{delJu} and actually for actions of all countable amenable groups by Foreman and Weiss \cite{ForWei}. It was therefore of surprise when Kechris and Rosendal showed in \cite{KeRo} that there is a generic action of the integers on the Cantor space. This result sparked such an interest that it has been since then re-proved several times by different authors, see e.g. \cite{AkGlWei}, \cite{BeDa}, \cite{Kwia}.

Another remarkable result of Hochman \cite{Hoch12} came few years later showing that this fails for $\Int^d$, for $d\geq 2$, i.e. these groups do not admit a generic action on the Cantor space. Hochman also attached the adjective `strong topological Rokhlin' to countable groups, as opposed to the previous usage when it was attached to topological groups or actions, to denote those countable groups that do admit a generic action on the Cantor space. Kwiatkowska in \cite{Kwia} later showed that free groups on finitely many generators do have this strong topological Rokhlin property in contrast to the result of Kechris and Rosendal from \cite{KeRo} that free groups on countably infinitely many generators do not.

The case of other countable groups has not been known and it is the aim of this paper to fill this gap. We show that the strong topological Rokhlin property, i.e. having a generic action on the Cantor space, is a property that is visible on the symbolic dynamical level. Indeed we show it is tightly connected with the structure of sofic subshifts over the corresponding group. Before continuing further, let us mention that symbolic dynamics is another area of dynamics traditionally reserved for actions of $\Int$ which has recently seen a tremendous progress in investigating more general group actions, where again substantial differences occur when the acting group varies. Early and by now classical results where these differences first occurred, between $\Int$ and $\Int^2$, are related to the domino problem and the existence of weakly and strongly aperiodic subshifts of finite type (see \cite{Berger} for these early results and see \cite{Cohen} for a recent breakthrough, where geometric group theory was beautifully blended into symbolic dynamics). Another important early occurrence is related to the Gottschalk surjunctivity conjecture (\cite{Gott}) which later gave rise to sofic groups (\cite{Grom}). We refer to \cite{CSCo-book} for general introduction into symbolic dynamics over countable groups.

The second main aim of this paper is to the strong topological Rokhlin property with the genericity of the shadowing property, also known as the pseudo-orbit tracing property. Shadowing is by now one of the fundamental notions in dynamical systems, closely related to hyperbolicity and topological stability. We refer to the monograph \cite{Palm} for an introduction and historical background. Although it was originally defined for single homeomorphisms (in fact, diffeomorphisms) it makes perfect sense for general discrete group actions and it was for the first time defined in this generality in \cite{OsTi} and investigated since then in numerous publications, see e.g. \cite{ChKeo}, \cite{Mey}, \cite{BarGaRaLi}, \cite{LiChZh}. Genericity of shadowing has been also already extensively studied, see e.g. \cite{PilPla}, \cite{BrMeRa}, \cite{Ko07}, \cite{KoMaOp}, \cite{BeDa}, and references therein. In \cite{BeDa}, genericity of shadowing for homeomorphisms on the Cantor space has been derived as the property of the generic integer action. We continue in this line of research for more general groups and discover a very strong relation with the strong topological Rokhlin property; in fact, an equivalence between the STRP and genericity of shadowing.

We start the presentation of our results. For the sake of the next theorem we informally define a \emph{projectively isolated subshift} $X\subseteq A^G$, where $G$ is a countable group and $A$ a finite set with at least two elements, as a closed subshift such that there exists a subshift of finite type $Y\subseteq B^G$, for some finite $B$, and a factor map $\phi: Y\to X$ such that $\phi[Y']=X$ for every subshift $Y'\subseteq Y$ that is sufficiently close to $Y$ with respect to the Hausdorff distance. A precise definition is provided later as Definition~\ref{def:projisolated}. Also for two subshifts $X,Y\subseteq A^G$ and a finite set $F\subseteq G$ we write $X\subseteq_F Y$ to denote that $X\subseteq Y$ and moreover the $F$-patterns of $X$ coincide with the $F$-patterns of $Y$.

\begin{thmx}\label{thm:intro1}
Let $G$ be a countable group. If $G$ is finitely generated, then the following are equivalent.

\begin{enumerate}
	\item\label{it:intro1-1} $G$ has the strong topological Rokhlin property.\medskip
	\item\label{it:intro1-2} For every closed subshift $X\subseteq A^G$, for some finite set with at least two elements, and every $\varepsilon>0$ there is a projectively isolated subshift $X'\subseteq A^G$ whose Hausdorff distance to $X$ is at most $\varepsilon$.\medskip
	\item\label{it:intro1-3} Shadowing is generic for continuous actions of $G$ on the Cantor space.\medskip
\end{enumerate}

For a general countable $G$, we have the equivalence between \eqref{it:intro1-1} and \eqref{it:intro1-2}, and moreover \eqref{it:intro1-3} implies \eqref{it:intro1-1} and \eqref{it:intro1-2}.
\end{thmx}
A sufficient condition that guarantees the strong topological Rokhlin property, and therefore also the genericity of shadowing for finitely generated groups, is that for every subshift of finite type $X\subseteq A^G$, for some finite $A$, and for every finite $F\subseteq G$ there are a subshift of finite type $Y\subseteq_F X$ and a finite set $E\subseteq G$ so that $Y$ is $\subseteq_E$-minimal.\medskip

We remark that there have been few similar results where the density of isolated points in some spaces was equivalent to genericity of certain actions, see \cite{GlKiMe} for actions of groups on countable sets (and another proof and related results in \cite{DouMa}) and \cite{KeLiPi} for representations of $C^*$-algebras and unitary representations of groups (see also \cite{DouMaVal} for another proof).\medskip

Theorem~\ref{thm:intro1} is then exploited to produce new examples and non-examples. From the positive side we show:
\begin{thmx}\label{thm:intro2}
Let $G=\bigstar_{i\leq n} G_i$ be a free product of the groups $(G_i)_{i\leq n}$, each of them being either finite or cyclic. Then $G$ has the strong topological Rokhlin property.
\end{thmx}

From the negative side we have more results and the following is a selection of some of them.
\begin{thmx}\label{thm:intro3}
Let $G$ be one of the following groups:
\begin{itemize}
	\item finitely generated infinite nilpotent group that is not virtually cyclic;
	\item $G_1\times G_2\times G_3$, where $G_i$, for $i\in\{1,2,3\}$ is finitely generated and recursively presented, and $G_1$ is indicable;
	\item an infinitely generated group such that for every finitely generated subgroup $H$ there is either $g\in G\setminus H$ that centralizes $H$, or has no relation with $H$ whatsoever.
\end{itemize}
Then $G$ does not have the strong topological Rokhlin property.
\end{thmx}

The paper is organized as follows. Section~\ref{sect:Cantor} consists of preliminaries and basic results on group actions on the Cantor space and symbolic dynamics over general groups. Most of the definition, crucial for the rest of the paper, are also contained there. Section~\ref{sect:mainproofs} contains the proof of the equivalence between \eqref{it:intro1-1} and \eqref{it:intro1-2} of Theorem~\ref{thm:intro1}. Section~\ref{sect:freeproducts} introduces new notions and techniques that are jointly with Theorem~\ref{thm:intro1} necessary to prove Theorem~\ref{thm:intro2}. Section~\ref{sect:noSTRP} is focused on providing new non-examples of groups with the strong topological Rokhlin property and the proof of Theorem~\ref{thm:intro3} is contained there, and Section~\ref{sect:genericdynamics} shows the equivalence of \eqref{it:intro1-3} with \eqref{it:intro1-2} of Theorem~\ref{thm:intro1}. Finally, in Section~\ref{sect:problems} we collect few open problems.\medskip

Let us also mention that our standard monographs concerning symbolic dynamics over general groups, topological dynamics of general group actions, and (geometric) group theory respectively are \cite{CSCo-book} (and also \cite{Bar17}), \cite{KeLi-book}, and \cite{DruKap-book} respectively, to which we refer in case the reader finds any unexplained notion from these areas in the sequel.
\section{Cantor and symbolic dynamics}\label{sect:Cantor}

Let $\CC$ denote the Cantor space. Denote by $\Homeo(\CC)$ the topological group of all self-homeomorphisms of $\CC$ equipped with the uniform topology.
\subsection{Spaces of actions and shifts}	
\subsubsection{Spaces of actions}
	\begin{definition}\label{def:spaceofactions}
For a countable group $G$, denote by $\Act_G(\CC)$ the Polish space of all continuous actions of $G$ on $\CC$. Formally, $\Act_G(\CC)$ is identified with the space of all homomorphisms from $G$ into the topological group $\Homeo(\CC)$, which can be further identified with a closed subset of $\Homeo(\CC)^G$ with the product topology.
	\end{definition}

Recall that by the Stone duality there is a one-to-one correspondence between homeomorphisms of $\CC$ and Boolean algebra isomorphisms of $\mathrm{Clopen}(\CC)$, the algebra of clopen sets of $\CC$. Notice also that this correspondence is a topological group isomorphism between $\Homeo(\CC)$ and $\mathrm{Aut}\big(\mathrm{Clopen}(\CC)\big)$, where the latter is equipped with the pointwise convergence topology. This observation immediately gives the following description of basic open subsets of $\Act_G(\CC)$ whose proof is left to the reader.
\begin{lemma}\label{lem:basicopennbhds}
Let $G$ be a countable group and $\alpha\in\Act_G(\CC)$. The following sets form basic open neighborhoods of $\alpha$: \[\NN_\alpha^{F,\PP}:=\big\{\beta\in\Act_G(\CC)\colon \forall f\in F\;\forall x\in\CC\;\forall P\in\PP\;\big(\alpha(f)x\in P\Leftrightarrow \beta(f)x\in P\big)\big\},\] where $F\subseteq G$ is a finite subset and $\PP$ is a partition of $\CC$ into disjoint non-empty clopen sets.
\end{lemma}

Notice that for any countable group $G$ the group $\Homeo(\CC)$ naturally acts on $\Act_G(\CC)$ by conjugation, where for $\phi\in\Homeo(\CC)$ and $\alpha\in\Act_G(\CC)$ the action $\phi\alpha\phi^{-1}$ is naturally defined by \[\big(\phi\alpha\phi^{-1}\big)(g):=\phi\alpha(g)\phi^{-1},\quad\text{for $g\in G$.}\] 

Although informally defined already in the abstract, for the sake of formal soundness let us provide a precise definition of the strong topological Rokhlin property.
\begin{definition}
Let $G$ be a countable group. We say that $G$ has the \emph{strong topological Rokhlin property} if there exists $\alpha\in\Act_G(\CC)$ such that the set \[\{\phi\alpha\phi^{-1}\colon \phi\in\Homeo(\CC)\}\subseteq \Act_G(\CC)\] is comeager.
\end{definition}
We state the following important fact that any invariant subset in $\Act_G(\CC)$ with the Baire property is either meager, or comeager.
\begin{fact}\label{fact:0-1-law}
Let $G$ be a countable group and let $\AAA\subseteq \Act_G(\CC)$ be a subset with the Baire property that is closed under conjugation, i.e. for any $\phi\in\Homeo(\CC)$, $\phi\AAA\phi^{-1}=\AAA$. Then $\AAA$ is either meager, or comeager.
\end{fact}
\begin{proof}
We apply \cite[Theorem 8.46]{KechrisBook} with $G$ equal to $\Homeo(\CC)$ and $X$ equal to $\Act_G(\CC)$. We only need to check that the action of $\Homeo(\CC)$ on $\Act_G(\CC)$ is topologically transitive. This is equivalent with the existence of an element $\alpha\in\Act_G(\CC)$ with a dense conjugacy class. This is proved in \cite[Proposition 1.2]{Hoch12}  - notice that the proposition is stated only for $\Int^d$, however the first paragraph of the proof mentions it works for any countable group.
\end{proof}
\subsubsection{Subshits and their spaces}
Let $A$ be a finite set with at least two elements and $G$ be a countable group. Consider the set $A^G$ equipped with the product topology with which it is either finite discrete if $G$ is finite, or homeomorphic to the Cantor space if $G$ is infinite. The group $G$ acts on $A^G$ by \emph{shift}, i.e. for $x\in A^G$ and $g,h\in G$ we have \[gx(h):=x(g^{-1}h).\] When we need a symbol for the shift action, we shall use $\sigma: G\curvearrowright A^G.$  Any closed and $G$-invariant subspace of $A^G$ will be called a \emph{subshift}.

If $F\subseteq G$ is a finite set, elements from $A^F$ will be called \emph{patterns}. If $X\subseteq A^G$ is a subshift, a pattern $p\in A^F$ is called \emph{forbidden} in $X$ if there is no $x\in X$ such that $x\upharpoonright F=p$; otherwise, the pattern $p$ is called \emph{allowed} in $X$.

Set \[X_F:=\{p\in A^F\colon \exists x\in X\; (x\upharpoonright F=p)\}\]

and for any pattern $p\in X_F$ we shall also set \[C_p(X):=\{x\in X\colon x\upharpoonright F=p\}.\] Then $\{C_p(X)\colon p\in X_F\}$ is a clopen partition of $X$. Sometimes, when the subshift $X$ is clear from the context, we may write just $C_p$ instead of $C_p(X)$.\medskip

By $\SH_G(A)$, we shall denote the compact space of all subshifts of $A^G$ equipped with the Vietoris topology - or equivalently, with the topology induced by the Hausdorff metric coming from a compatible metric on $A^G$, and by $\SH_G(n)$ we denote the space $\SH_G(\{1,\ldots,n\})$. When there is no danger of confusion, we may identify the spaces $\SH_G(A)$ and $\SH_G(n)$ for $|A|=n$.


In the sequel, we shall use the wording `non-trivial finite set' to emphasize that the set in question has at least two elements.\medskip

There is a convenient form of basic open neighborhoods of subshifts in $\SH_G(A)$. The proof of the following lemma is similar to the proof of Lemma~\ref{lem:basicopennbhds} and left to the reader.
\begin{lemma}
	Let $G$ be a countable group, $A$ a non-trivial finite set, and $X\subseteq A^G$ a subshift. Basic open neighborhoods of $X$ in $\SH_G(A)$ are of the form \[\NN_X^F:=\{Y\subseteq A^G\colon Y_F=X_F\},\] where $F\subseteq G$ is a finite set.
\end{lemma}
Notice that the set of the form $\NN_X^F$ as above is actually clopen in $\SH_G(A)$.
\begin{remark}
Since the spaces of subshifts $\SH_G(A)$ will play a major role in this paper, we feel obliged to provide few comments on them. They closely resemble the spaces of subgroups of a given countable group with the Chabauty topology (see \cite{Cha}) and surely any reader familiar with the latter will also feel comfortable with the former. We could not track the origin of when these spaces appeared for the first time and apparently at least for the group $\Int$, and also for $\Int^d$, they have been in a folklor use for some time. For more general groups the oldest references we could find were \cite{GaoJacSew} and \cite{FriTam}. Since then these spaces, for general groups, e.g. played a major role in characterizing strongly amenable countable groups in \cite{FriTamVF}.
\end{remark}

\begin{construction}
Let $\PP=\{P_1,\ldots,P_n\}$ be a partition of a compact metrizable zero-dimensional space $X$ into disjoint non-empty clopen sets and let $\alpha: G\curvearrowright X$ be a continuous action of a countable group $G$ on $X$. We denote by $Q^\alpha_\PP$ the continuous $G$-equivariant map from $X$ to $\PP^G$, freely identified with $n^G$, which is defined as follows. For any $x\in X$, $g\in G$ and $i\leq n$ \[Q^\alpha_\PP(x)(g):= i\;\;\;\text{ if and only if }\;\;\; \alpha(g^{-1})x\in P_i.\]
\end{construction}

The verification that $Q^\alpha_\PP$ is continuous and $G$-equivariant is straightforward.

By $\QQ(\alpha,\PP)$ we shall denote the subshift of $n^G$, an element of $\SH_G(n)$, which is the image of $Q^\alpha_\PP$.

The following basic lemma, which is a generalization of the well-known Curtis-Hedlund-Lyndon theorem from symbolic dynamics, shows that every continuous equivariant map from a zero-dimensional compact metrizable space $X$, equipped with an action $\alpha: G\curvearrowright X$, onto a subshift is of the form $Q^\alpha_\PP$, for some clopen partition $\PP$ of $X$.

\begin{lemma}\label{lem:factoringonshift}
Let $G$ be a countable group acting continuously on a zero-dimensional compact metrizable space $X$. Let us denote the action by $\alpha$. Let $\phi: (X,\alpha)\rightarrow A^G$ be a continuous $G$-equivariant map into the shift space $A^G$, for some non-trivial finite $A$. Then there is a clopen partition $\PP$ of $X$ such that $\phi=Q^\alpha_\PP$.
\end{lemma}
\begin{proof}
Let $Y\subseteq A^G$ be the image of $X$ via $\phi$, which is a subshift. Enumerate $A$ as $\{a_1,\ldots,a_n\}$ and assume without loss of generality that for every $i\leq n$ there is $y\in Y$ with $y(1_G)=a_i$. We leave to the reader to check that the clopen partition \[\PP:=\Big\{\phi^{-1}\big(\{y\in Y\colon y(1_G)=a_i\}\big)\colon i\leq n\Big\}\] is as desired.
\end{proof}
The following standard result can be derived as a corollary.
\begin{corollary}[The Curtis-Hedlund-Lyndon theorem]
	Let $G$ be a countable group, $A$ and $B$ be two non-trivial finite sets, and $\phi:X\rightarrow Y$ be a continuous $G$-equivariant map between two subshifts $X\subseteq A^G$ and $Y\subseteq B^G$. Then there exist a finite set $F\subseteq G$ and a map $f:X_F\rightarrow B$ such that for every $x\in X$ and $g\in G$ \[\phi(x)(g)=f(g^{-1}x\upharpoonright F).\]
\end{corollary}
The following simple lemma will also prove to be useful in the sequel.
\begin{lemma}\label{lem:SFTfactorofSFT}
	Let $A$ be a finite set with at least two elements, $G$ a countable group and let $X\subseteq A^G$ be a subshift. Let $F\subseteq G$ be a finite subset. Then for the partition \[\PP:=\{C_p\subseteq X\colon p\in X_F\},\] the map $Q^\sigma_\PP$, where $\sigma$ is the shift action on $X$, is one-to-one.
\end{lemma}
\begin{proof}
	The image is clearly a closed $G$-invariant subset of $\PP^G$, i.e. a subshift, so we only need to check the map is injective. This is however obvious since if $x\neq y\in X$ then there is $g\in G$ such that $\sigma(g)x$ and $\sigma(g)y$ lie in different elements of the partition $\PP$ and thus $Q^\sigma_\PP(x)(g^{-1})\neq Q^\sigma_\PP(y)(g^{-1})$.
\end{proof}
We continue with two propositions connecting the spaces and topologies of $\Act_G(\CC)$ and $\SH_G(n)$ which will be instrumental in proving the main theorem.
\begin{proposition}\label{prop:Qcontinuity}
Let $G$ be a countable group and $\PP$ a partition of $\CC$ into disjoint non-empty clopen sets. Then the map \[\QQ(\cdot,\PP):\Act_G(\CC)\to \SH_G(\PP)\] defined as \[\alpha\to \QQ(\alpha,\PP)\] is continuous and onto the set $\NN_{\PP^G}^{\{1_G\}}=\{X\in\SH_G(\PP)\colon \forall P\in\PP\;\exists x\in X\;(x(1_G)=P)\}$ consisting of all subshifts of $\PP^G$ that contain each letter from $\PP$.
\end{proposition}
\begin{proof}
We fix $G$ and $\PP$ as in the statement. Take $\alpha\in\Act_G(\CC)$ and set $X:=\QQ(\alpha,\PP)$. Let $U$ be an open neighborhood of $X$ which we may assume is of the form $\NN_X^F$, for some finite symmetric set $F\subseteq G$, which, without loss of generality, contains the unit. We shall find an open neighborhood $V$ of $\alpha$ so that $\QQ(\cdot,\PP)[V]\subseteq U$. We claim that $V:=\NN_\alpha^{F,\PP}$ is such a neighborhood. Indeed, pick $\beta\in V$ and set $Y:=\QQ(\beta,\PP)$. Let us show that $Y\in U=\NN_X^F$, i.e. $Y_F=X_F$. Let $p\in X_F$ be a pattern allowed in $X$, so there is $x\in X$ such that $x\upharpoonright F=p$. Let $z\in\CC$ be an arbitrary element such that $Q^\alpha_\PP(z)=x$. By the definition of the neighborhood $V=\NN_\alpha^{F,\PP}$ \[\forall f\in F\;\forall P\in\PP\; \big(\alpha(f)z\in P\Leftrightarrow \beta(f)z\in P\big).\] It follows that for $y:=Q^\beta_\PP(z)\in Y$ we have \[p=x\upharpoonright F=y\upharpoonright F,\] thus $p$ is allowed in $Y$ as well. We proved that $X_F\subseteq Y_F$, the other direction is proved symmetrically. This finishes the proof that $\QQ(\cdot,\PP)$ is continuous.

Let us show that the map is onto $\NN_{\PP^G}^{\{1_G\}}$. Pick $X\in \NN_{\PP^G}^{\{1_G\}}$ and and for each $P\in\PP$, let $\psi_P:C_{\{1_G,P\}}(X)\times\CC\rightarrow P$ be a homeomorphism. Here $C_{\{1_G,P\}}(X):=\{x\in X\colon x(1_G)=P\}$, which corresponds to the notation $C_p(X)$, where $p=\{1_G,P\}\in \PP^{\{1_G\}}$. We take the product of $C_{\{1_G,P\}}(X)$ with $\CC$ to ensure that it has no isolated points. 

Since $\{C_{\{1_G,P\}}(X)\times\CC\colon P\in\PP\}$ is a clopen partition of $X\times\CC$, it follows that $\psi:=\coprod_{P\in\PP} \psi_P: X\times\CC\rightarrow\CC$ is a homeomorphism. We define $\gamma:=\psi\circ\big(\sigma\times\mathrm{Id}\big)\circ\psi^{-1}\in\Act_G(\CC)$, where $\sigma\times\mathrm{Id}$ is the action of $G$ on $X\times\CC$, which acts as the shift on the first coordinate and as the identity on the other. We claim that $\QQ(\gamma,\PP)=X$. This follows since $\psi$ is a homeomorphism and for every $x\in X$ and $y,z\in\CC$ such that $z=\psi(x,y)$, and for every $g\in G$ and $P\in\PP$ we have \[\begin{split}Q_\PP^\gamma(z)(g)=P& \Leftrightarrow \gamma(g^{-1})z\in P\\ &\Leftrightarrow \big(\sigma(g^{-1})\times\mathrm{Id}\big)(x,y)\in C_{\{1_G,P\}}(X)\times\CC\Leftrightarrow x(g)=P.\end{split}\]
\end{proof}

\begin{proposition}\label{prop:Qmap-nbhds}
Let $G$ be a countable group and $\PP'\preceq\PP$ two clopen partitions of $\CC$, one refining the other. 
\begin{enumerate}
	\item\label{it1-Qmap-nbhds} For every $\alpha\in\Act_G(\CC)$, finite symmetric set $F\subseteq G$ containing $1_G$ and $X\in \QQ(\cdot,\PP)[\NN_\alpha^{F,\PP}]$ we have \[\QQ(\cdot,\PP)[\NN_\alpha^{F,\PP}]=\NN_X^F.\]
	
	\item\label{it2-Qmap-nbhds} For every $\alpha\in\Act_G(\CC)$ we have $\QQ(\alpha,\PP)=\phi\big(\QQ(\alpha,\PP')\big)$, where $\phi:(\PP')^G\rightarrow \PP^G$ is the map induced by the inclusion map $\iota:\PP'\rightarrow\PP$, i.e. satisfying $P\subseteq \iota(P)$ for every $P\in\PP'$. 
	
	\item\label{it3-Qmap-nbhds} For every $\alpha\in\Act_G(\CC)$, a finite symmetric set $F\subseteq G$ containing $1_G$, and a subshift $Y\subseteq (\PP')^G$ such that $\phi[Y]=\QQ(\alpha,\PP)$, where $\phi:(\PP')^G\to \PP^G$ is a map induced by the inclusion map $\iota:\PP'\to\PP$, there exists $\beta\in\NN_\alpha^{F,\PP}$ such that $\QQ(\beta,\PP')=Y$.
\end{enumerate}
\end{proposition}
\begin{proof}
Let us fix $G$ and $\PP'\preceq\PP$ as in the statement.\medskip

We first prove \eqref{it1-Qmap-nbhds}. Fix additionally $\alpha$, $F\subseteq G$ and $X$ as in the statement. Since for every $Y\in\NN_X^F$ we have $\NN_Y^F=\NN_X^F$, we can without loss of generality assume that $X=\QQ(\alpha,\PP)$. We first prove that $\QQ(\cdot,\PP)[\NN_\alpha^{F,\PP}]\subseteq \NN_X^F$. Pick $\beta\in\NN_\alpha^{F,\PP}$ and set $Y:=\QQ(\beta,\PP)$. We need to prove that $X_F=Y_F$. Take some $p\in X_F$ and let $x\in\CC$ be such that for every $f\in F$ and $P\in\PP$ \[\alpha(f^{-1})x\in P\Leftrightarrow p(f)=P,\] thus $Q_\PP^\alpha(x)\upharpoonright F=p$. Since by definition for every $f\in F$ and $P\in\PP$ we have \[\alpha(f^{-1})x\in P\Leftrightarrow \beta(f^{-1})x\in P,\] it follows that $Q_\PP^\beta(x)\upharpoonright F=p$ as well. We showed that $X_F\subseteq Y_F$, the inclusion $Y_F\subseteq X_F$ is proved symetrically.

Now we prove the reverse inclusion $\NN_X^F\subseteq \QQ(\cdot,\PP)[\NN_\alpha^{F,\PP}]$. Take any $Y\in\NN_X^F$. We define a clopen partition $\PP'':=\{R_p\colon p\in X_F\}\preceq\PP$, where for $p\in X_F$, \[R_p:=\{x\in\CC\colon \forall f\in F\; \big(\alpha(f^{-1})(x)\in P\;\text{ if and only if }\; p(f)=P\big)\}.\] For each $p\in X_F$, since $Y_F=X_F$, we have that $C_p(Y)\neq\emptyset$ and let $\psi_p:C_p(Y)\times\CC\rightarrow R_p$ be a fixed homeomorphism and set $\psi:Y\times\CC\rightarrow \CC$ to be the homeomorphism $\coprod_{p\in X_F} \psi_p$. We define the action $\beta\in\Act_G(\CC)$ by setting for $g\in G$ and $x\in \CC$ \[\beta(g)(x):=\psi\circ\big(\sigma(g)\times\mathrm{Id}\big)\circ\psi^{-1}(x).\]

The verification that $\QQ(\beta,\PP)=Y$ is straightforward and similar as in the proof of Proposition~\ref{prop:Qcontinuity}. We check that $\beta\in \NN_\alpha^{F,\PP}$. Pick $x\in\CC$, $f\in F$, and $P\in\PP$. We need to verify that $\alpha(f^{-1})x\in P$ if and only if $\beta(f^{-1})x\in P$. Without loss of generality, assume that $\alpha(f^{-1})x\in P$ and it suffices now to verify that $\beta(f^{-1})x\in P$. Let $p\in X_F$ be such that $x\in R_p$. Then by definition, $p(f)=P$ and we verify that $\beta(f^{-1})x=\psi\circ (\sigma(f^{-1})\times\mathrm{Id})\circ\psi^{-1}(x)\in P$. We have $\psi^{-1}(x)\in C_p(Y)\times \CC$ and so $(\sigma(f^{-1})\times\mathrm{Id})\circ\psi^{-1}(x)\in C_{p'}(Y)\times\CC$, for some $p'\in X_F$ where $p'(1_G)=P$. Consequently, $\beta(f^{-1})x\in R_{p'}$, thus, by the definition of $R_{p'}$, \[\beta(f^{-1})x=\alpha(1^{-1}_G)\big(\beta(f^{-1})x\big)\in p'(1_G)=P,\] which is what we were supposed to show.\medskip

We now prove \eqref{it2-Qmap-nbhds}. Fix $\alpha\in\Act_G(\CC)$ and consider the map $\phi:(\PP')^G\rightarrow\PP^G$ as in the statement. Set $X:=\QQ(\alpha,\PP')$ and $Y:=\QQ(\alpha,\PP)$. We need to show that $\phi[X]=Y$. Pick an arbitrary $z\in\CC$ and set $x:=Q_{\PP'}^\alpha(z)\in X$, $y:=Q_\PP^\alpha(z)\in Y$. Now let $g\in G$ and we check that $\phi(x)(g)=y(g)$, which will finish the proof. For $P\in\PP$ we have \[\begin{split}y(g)=P & \Leftrightarrow \alpha(g^{-1})z\in P\Leftrightarrow \exists P'\in\PP'\;\big( P'\subseteq P\wedge \alpha(g^{-1})z\in P'\big)\\ & \Leftrightarrow x(g)=P'\Leftrightarrow \phi(x)(g)=P,\end{split}\] which shows the equality.\medskip

Finally, we prove \eqref{it3-Qmap-nbhds}. Fix $\alpha$, $F\subseteq G$, and $Y$ as in the statement. Set $X:=\QQ(\alpha,\PP)$. The proof is similar to the proof of \eqref{it1-Qmap-nbhds}. We again define a clopen partition $\PP'':=\{R_p\colon p\in X_F\}\preceq\PP$, where for $p\in X_F$, \[R_p:=\{x\in\CC\colon \forall f\in F\; \big(\alpha(f^{-1})(x)\in P\;\text{ if and only if }\; p(f)=P\big)\}.\] Consider now the clopen partition refining $\PP'$ and $\PP''$, i.e. the partition $\{R_p\cap P\colon p\in X_F, P\in\PP', R_p\cap P\neq\emptyset\}$. For every $p\in X_F$ and $P\in\PP'$ we also define \[D_p:=\big\{y\in Y\colon \forall f\in F\;\big(\phi_0(y(f))=p(f)\big)\big\}\;\text{and}\; C_P:=\{y\in Y\colon y(1_G)=P\}.\] Since $\phi[Y]=X$ we get that for each $p\in X_F$, $D_p\neq\emptyset$, and $\{D_p\colon p\in X_F\}$ is a clopen partition of $Y$. Consequently, $\{D_p\cap C_P\colon p\in X_F,P\in\PP',R_p\cap P\neq\emptyset\}$ is a clopen partition of $Y$ as well. As in the proof of \eqref{it1-Qmap-nbhds}, for every $p\in X_F$ and $P\in\PP'$ such that $R_p\cap P\neq\emptyset$, let $\psi_{p,P}: (D_p\cap C_P)\times\CC\to R_p\cap P$ be a fixed homeomorphism. We set $\psi:Y\times\CC\to\CC$ to be the homeomorphism \[\coprod_{p\in X_F,P\in\PP',R_p\cap P\neq\emptyset}\psi_p\] and we define the action $\beta\in\Act_G(\CC)$ by setting for $g\in G$ and $x\in \CC$ \[\beta(g)(x):=\psi\circ\big(\sigma(g)\times\mathrm{Id}\big)\circ\psi^{-1}(x).\] The verification that $\beta\in\NN_\alpha^{F,\PP}$ is as in the proof of \eqref{it1-Qmap-nbhds}. To check that $\QQ(\beta,\PP')=Y$, notice that for every $x\in\CC$, $g\in G$, and $P\in\PP'$ we have $\beta(g^{-1})(x)\in P$ if and only if $y(g)=P$, where $\psi(y,z)=x$ for some $z\in\CC$. Thus $Q_{\PP'}^\beta(x)=y$ and by definition, $y\in Y$. It follows that $\QQ(\beta,\PP')\subseteq Y$. Conversely, for every $y\in Y$, pick any $z\in\CC$ and set $x:=\psi(y,z)$. It is straightforward and left to the reader that $Q_{\PP'}^\beta(x)=y$, thus $Y\subseteq \QQ(\beta,\PP')$. 
\end{proof}

\subsection{Isolated subshifts, subshifts of finite type, and sofic subshifts}
\begin{definition}\label{def:SFT}
	Let $G$ be a countable group and $A$ be a finite set with at least two elements. A subshift $X\subseteq A^G$ is \emph{of finite type}, shortly an \emph{SFT}, if there exists a finite set $F\subseteq G$, called the \emph{defining window} of $X$, and a set $\FF\subseteq A^F$ of patterns such that for $x\in A^G$ we have \[x\in X\;\text{ if and only if }\; \forall g\in G\; (gx\upharpoonright F\in\FF).\]
	The set $\FF$ is called \emph{the set of allowed patterns} for $X$, while its complement $F^A\setminus \FF$ is called \emph{the set of forbidden patterns} for $X$.
	
	We also say that a subshift $X\subseteq A^G$ is \emph{sofic} if it is a factor of a subshift of finite type (possibly defined over different finite set, i.e. subshift of $B^G$ for some non-trivial finite $B$).
\end{definition}

Subshifts of finite type play also a prominent role in the topology of the space $\SH_G(A)$. The proof of the following simple but useful lemma is left to the reader.
\begin{lemma}\label{lem:SFTnbhrds}
Let $G$ be a countable group and $A$ be a non-trivial finite set. Then the subshifts of finite type are dense in $\SH_G(A)$. Moreover, for every subshift of finite type $X\subseteq A^G$ there is a finite set $F\subseteq G$ so that all subshifts in the basic open neighborhood $\NN_X^F$ are subshifts of $X$.

In particular, every open set in $\SH_G(A)$ contains an open subset which has a subshift that is maximal with respect to inclusion, it is of finite type, and all other subshifts in the open set are its subshifts.
\end{lemma}

\begin{definition}\label{def:Rauzy}
Let $G$ be a countable group, $F\subseteq G\setminus\{1_G\}$ a finite subset, and $V$ a finite (vertex) set. We call the collection $\VV_F=(V,(E_f)_{f\in F})$ \emph{an $F$-Rauzy graph} if for each $f\in F$, $(V,E_F)$ is a directed graph with no sources and no sinks. That is, $E_f$ is a set of oriented edges between the vertices $V$ such that for every $v\in V$ there are at least one incoming and one outgoing edge to $v$, resp. from $v$.

Having $\VV_F$ as above we define a subshift $X_{\VV_F}\subseteq V^G$ as follows. For $x\in V^G$ we set \[x\in X_{\VV_F}\;\text{ if and only if }\; \forall g\in G\;\forall f\in F\;\big((x(g),x(gf))\in E_f\big).\]
It is clear that $X_{\VV_F}$ is of finite type.
\end{definition}

Notice that alternatively and equivalently one can require to have the edge set $E_f$ for all $f\in G$ with the requirement that for all but finitely many $f\in G$, $E_f$ is a complete directed graph on $V$.\medskip

Conversely, it is well-known fact (at least in case $G=\Int$) that any SFT $X\subseteq A^G$, for any $G$ and $A$, is conjugate to an SFT of the form $X_{\VV_F}$ for some $F$-Rauzy graph. We refer to \cite[Proposition 1.6]{Bar17} for the proof of the following proposition.
\begin{proposition}\label{prop:SFTprop}
Let $G$ be a countable group and $A$ a finite set with at least two elements. Let $X\subseteq A^G$ be a subshift of finite type whose defining window is a finite set $F\subseteq G$. Let $S$ be a finite symmetric set (not containing $1_G$) such that $\langle F\rangle=\langle S\rangle\leq G$. Then there exists an $S$-Rauzy graph $\VV_S$, 
such that $X$ is conjugate to $X_{\VV_S}$.
\end{proposition}
We remark that the isomorphism between $X_{\VV_S}$ and $X$ may be by construction chosen so that it is induced by a map $f:V\to A$, where $V$ is the vertex set of $\VV_S$.\medskip

A special subclass of subshifts of finite type will be of crucial importance.
\begin{definition}\label{def:isolatedsubshift}
	Let $G$ and $A$ be as in Definition~\ref{def:SFT}. A subshift $X\in\SH_G(A)$ is called \emph{isolated} if any of the following equivalent conditions hold:
	\begin{enumerate}
		\item $X$ is isolated in the topology of $\SH_G(A)$.
		\item\label{def:isolatedsubshift2} $X$ is of finite type and there exists a finite set $F\subseteq G$ such that there is no proper subshift $Y\subsetneq X$ satisfying $Y_F=X_F$.
	\end{enumerate}
\end{definition}

\begin{remark}
Notice that the implication from (1) to (2) follows from Lemma~\ref{lem:SFTnbhrds}. The same lemma implies that if $X$ is not isolated and is of finite type there exist a finite set $F\subseteq G$ and $Y\in\NN_X^F\setminus\{X\}$ such that $Y\subseteq X$, thus $Y\subsetneq X$ while $Y_F=X_F$, hereby showing the implication from (2) to (1).

Notice also that every minimal subshift of finite type is isolated.
\end{remark}
\begin{lemma}\label{lem:isolatedcharacterization}
Let $A$ be a finite set having at least two elements and $G$ a countable group. A subshift $X\subseteq A^G$ is isolated if and only if it is of finite type and there exists a partition $\PP$ of $X$ into disjoint non-empty clopen sets such that there is no proper non-empty subshift $Y\subseteq X$ which intersects every element of $\PP$.
\end{lemma}
\begin{proof}
If $X$ is isolated then by definition it is of finite type and there is a finite set $F\subseteq G$ so that $\NN_X^F=\{X\}$. Then the partition \[\PP:=\{C_p\subseteq X\colon p\in A^F\text{ is an allowed pattern in }X\}\] is as desired.

Conversely, suppose that $X\subseteq A^G$ is of finite type and has the property as in the statement with respect to a partition $\PP$. It is clear that then it has the same property with respect to any refinement $\PP'\preceq \PP$. Thus we can find a refinement $\PP'\preceq \PP$ which is of the form $\{C_p\colon p\in A^F\text{ is an allowed pattern in }X\}$, for some finite set $F\subseteq G$. It is clear that then $\NN_X^F=\{X\}$, so $X$ is isolated.
\end{proof}
The following corollary that being isolated is a conjugacy invariant immediately follows from Lemma~\ref{lem:isolatedcharacterization} and the fact that being of finite type is a conjugacy invariant (see e.g. \cite[Proposition 1.5]{Bar17}).
\begin{corollary}\label{cor:isolatedinvariant}
Let $A$ and $B$ be finite sets with at least two elements, and $G$ a countable group. Suppose that $X\subseteq A^G$ is isolated and $Y\subseteq B^G$ is a subshift such that there is a $G$-equivariant homeomorphism between $X$ and $Y$. Then $Y$ is isolated as  well.
\end{corollary}

\begin{lemma}\label{lem:CHL}
Let $A$ and $B$ be two finite sets with at least two elements and let $G$ be a countable group. Let $X\in\SH_G(A)$ and $Y\in\SH_G(B)$ be two shifts and $\phi:X\rightarrow Y$ a continuous $G$-equivariant map. Then there exists a finite set $F\subseteq G$ such that for every $Z\subseteq Z'\in\NN_X^F$, $\phi$ is defined on $Z$.
\end{lemma}
\begin{proof}
Fix $X$, $Y$ and $\phi:X\rightarrow Y$ as in the statement. By the Curtis-Hedlund-Lyndon theorem, $\phi$ is induced by some map $\phi_0: X_F\rightarrow B$, where $F\subseteq G$ is a finite subset, so that for every $x\in X$ and $g\in G$ we have \[\phi(x)(g):=\phi_0\big(g^{-1}x\upharpoonright F\big).\] Now if $Z\subseteq Z'\in\NN_X^F$ then $Z_F\subseteq Z'_F=X_F$, thus $\phi$ can be defined on $Z$ using $\phi_0$ exactly in the same way.
\end{proof}
The previous lemma makes sense of the following definition.
\begin{definition}\label{def:projisolated}
Let $A$ be a finite set with at least two elements and $G$ be a countable group. A subshift $X\in\SH_G(A)$ is called \emph{projectively isolated} if there exist a subshift $Y\in\SH_G(B)$, for some finite set $B$, a finite set $F\subseteq G$, and a continuous $G$-equivariant map $\phi: Y\rightarrow X$ such that the set $\{Z\in \NN_Y^F\colon \phi[Z]=X\}$ has a non-empty interior containing $Y$.

We shall call any such map $\phi_Z$ an \emph{isolated factor map}.

If $X$ is a factor of an isolated subshift, then we call $X$ \emph{strongly projectively isolated}.
\end{definition}

We collect several basic observations and lemmas about projectively isolated subshifts.

\begin{lemma}
If $X\in\SH_G(A)$ is an isolated subshift, then it is also strongly projectively isolated, and a strongly projectively isolated subshift is projectively isolated.
\end{lemma}
\begin{proof}
Let $X$ be strongly projectively isolated, i.e. it is a factor of some isolated subshift $Y$. Using the notation of Definition~\ref{def:projisolated}, we take $\NN_Y^F$ to be the neighborhood $\{Y\}$ and we take the identity as $\phi$.

If $X$ is isolated, then it is obviously a factor of an isolated subshift, thus strongly projectively isolated.
\end{proof}

\begin{lemma}
Every projectively isolated subshift is sofic. In particular, there are at most countably many projectively isolated subshifts - for each alphabet.
\end{lemma}
\begin{proof}
Suppose that $X$ is projectively isolated. Then by definition there is an open set of subshifts all of them projecting onto $X$. Since by Lemma~\ref{lem:SFTnbhrds} subshifts of finite type are dense, we get that $X$ is a factor of a subshift of finite type, thus it is sofic.
\end{proof}

The following lemma is clear.
\begin{lemma}\label{lem:projshiftproperties}
Let $X$ be a projectively isolated subshift and let $Y$ be a subshift witnessing that $X$ is projectively isolated - as in Definition~\ref{def:projisolated}. Then every subshift that is a factor of $X$ is also projectively isolated and every subshift that factors onto $Y$ also witnesses that $X$ is projectively isolated.
\end{lemma}
Since an isomorphism is a special case of a factor map, Lemma~\ref{lem:projshiftproperties} immediately gives.
\begin{corollary}\label{cor:projisolatedconjugacy}
Being projectively isolated is a conjugacy invariant.
\end{corollary}

The content of the following lemma is that the map projecting onto a projectively isolated subshift is without loss of generality induced by a map between their respective alphabets.
\begin{lemma}\label{lem:standardformoffactormaps}
Let $X\subseteq A^G$, for some $A$ and $G$, be a projectively isolated subshift and let $Y$ be a subshift factoring onto $X$ witnessing that. Then there exists a subshift $Y'\subseteq B^G$ isomorphic to $Y$ which witnesses that $X$ is projectively isolated via a map $P:Y'\rightarrow X$ which is defined by a map $P_0:B\to A$.
\end{lemma}
\begin{proof}
Let $X$ and $Y$ be as in the statement and suppose that the projective isolatedness is witnessed by a map $\phi:Y\rightarrow X$ induced by a map $\phi_0:Y_F\rightarrow A$ between the alphabets, where $F\subseteq G$ is finite. Consider the partition $\PP:=\{C_p\subseteq Y\colon p\in Y_F\}$ of $Y$ as in Lemma~\ref{lem:SFTfactorofSFT}. Using the notation of Lemma~\ref{lem:SFTfactorofSFT}, $\QQ_\PP^\sigma: Y\rightarrow \PP^G$ induces an isomorphism between $Y$ and its image denote $Y'\subseteq B^G$, where $B$ is identified with $\PP$. It is then clear that the map $P:=\phi\circ (\QQ_\PP^\sigma)^{-1}: Y'\rightarrow X$ is induced by a map $P_0:B\to A$, finishing the proof.
\end{proof}
\begin{example}
Let us provide an example of a subshift that is strongly projectively isolated, however it is not isolated. Let $X\subseteq \{0,1\}^\Int$ be the shift \[\{x\in \{0,1\}^\Int\colon |x^{-1}(\{1\})|\leq 1\}.\] $X$ cannot be isolated since it is not even a subshift of finite type. Let $Y\subseteq \{-1,0,1\}^\Int$ be a subshift of finite type whose defining window is an interval of length $2$ and the allowed patterns are $\{-1-1,-11,10,00\}$. Let $P:Y\rightarrow X$ be a factor map induced by the map $P_0:\{-1,0,1\}\to\{0,1\}$ defined by $P_0(-1)=P_0(0)=0$ and $P_0(1)=1$. One easily checks that $\NN_Y^{\{0,1\}}=\{Y\}$, i.e. $Y$ is isolated, thus $Y$ and $P$ witness that $X$ is projectively isolated.
\end{example}

We conclude this section with one more notion of a subshift, weaker than being sofic, yet still obtained given finite data.
\begin{definition}\label{def:effectiveshubshift}
Let $G$ be a finitely generated and recursively presented group and let $A$ be a non-trivial finite set. Fix a finite symmetric generating set $S\subseteq G$ for $G$. We say that a subshift $X\subseteq A^G$ is \emph{effective} if there exists an algorithm (formally, a Turing machine) that given a pattern $p\in A^F$, which is presented to the algorithm in the form $\{(w_i,a_i)\colon i\leq n\}$, where each $a_i\in A$ and $w_i$ is a word in letters from $S$, decides whether $p$ is allowed in $X$, or not.
\end{definition}
\begin{remark}
The notion of an effective subshift was introduced by Hochman for $\Int^d$-subshifts in \cite{Hoch09}. A general notion for general finitely generated groups was given in \cite{AuBarSab}. The definition above corresponds to their definition jointly with \cite[Proposition 2.1 and Lemma 2.3]{AuBarSab} for recursively presented groups. In particular, notice that, without loss of generality, the set of forbidden patterns can be taken recursive instead of only recursively enumerable. We will restrict to this case in this paper.
\end{remark}
The following follows from the definition.
\begin{lemma}\label{lem:effectivesubshiifts}
Let $1\to N\to G\to H\to 1$ be a short exact sequence of groups where $G$ and $H$ are finitely generated and recursively presented. We assume that $N\subseteq G$. Let $A$ be a non-trivial finite set, let $X\subseteq A^G$ be a subshift on which $N$ acts trivially, and let $X'\subseteq A^H$ be the corresponding subshift over $H$. Then $X$ is effective if and only if $X'$ is.
\end{lemma}

\section{Strong topological Rokhlin property}\label{sect:mainproofs}
This section is devoted to proving the equivalence between \eqref{it:intro1-1} and \eqref{it:intro1-2} of Theorem~\ref{thm:intro1} from Introduction. We restate it more precisely below.
\begin{theorem}\label{thm:mainRokhlin}
Let $G$ be a countable group. Then $G$ has the strong topological Rokhlin property if and only if for every $n\geq 2$, the set of projectively isolated subshifts in $\SH_G(n)$ is dense.
\end{theorem}

We prove the two implications separately. 
\subsection{Failure of the strong topological Rokhlin property}
\begin{proposition}\label{prop:noRokhlin}
Let $G$ be a countable group. If there exists $n\geq 2$ such that the projectively isolated subshifts are not dense in $\SH_G(n)$, then $G$ does not have the strong topological Rokhlin property. In fact, every conjugacy class in $\Act_G(\CC)$ is meager.
\end{proposition}
\begin{proof}
We fix a countable group $G$ without the STRP.

Let $n\geq 2$ be such that the set of projectively isolated subshifts is not dense in $\SH_G(n)$. That is, there is a non-empty open set $U\subseteq \SH_G(n)$ that contains no projectively isolated subshifts. By passing to an open subset of $U$ if necessary, we may without loss of generality assume that there is a subset $I\subseteq \{1,\ldots,n\}$ such that for every $X\in U$ we have $\{i\leq n\colon \exists x\in X\;(x(1_G)=i)\}=I$, and by passing to $\SH_G(m)$, for some $m<n$, if necessary, we may without loss of generality assume that $I=\{1,\ldots,n\}$. For every $X\in U$ set \[\begin{split} A(X):= \{ & \alpha\in \Act_G(\CC)\colon\\ & \text{for no clopen partition } \PP=\{P_1,\ldots,P_n\}\text{ of }\CC,\; \QQ(\alpha,\PP)= X\}.\end{split}\]

\begin{lemma}\label{lem:claimlemma}
For every $X\in U$, $A(X)$ is a dense $G_\delta$ set.
\end{lemma}

\begin{proof}[Proof of Lemma~\ref{lem:claimlemma}]
Fix a clopen partition $\PP=\{P_1,\ldots,P_n\}$ and we show that the set \[\AAA^\PP_X:=\{\alpha\in\Act_G(\CC)\colon \QQ(\alpha,\PP)\neq X\}\] is open and dense. In order to show that it is open, we use Proposition~\ref{prop:Qcontinuity} telling us that the map $\QQ(\cdot,\PP)$ is continuous. It follows that \[\AAA^\PP_X=\QQ^{-1}(\cdot,\PP)\Big(\SH_g(n)\setminus\{X\}\Big)\] is open as a preimage of an open set.

We now show that $\AAA^\PP_X$ is dense. Fix some open set $V\subseteq \Act_G(\CC)$. We claim that there is $\alpha\in V$ such that $\QQ(\alpha,\PP)\neq X$. 

We may suppose that $V$ is of the form $\NN_\beta^{F,\PP'}$, for some $\beta\in\Act_G(\CC)$, finite symmetric $F\subseteq G$ containing the unit $1_G$, and a clopen partition $\PP'\preceq\PP$. By Proposition~\ref{prop:Qmap-nbhds}~\eqref{it1-Qmap-nbhds}, setting $Y:=\QQ(\beta,\PP')$, we have $\QQ(\cdot,\PP')[V]=\NN_Y^F$. Let $\phi:(\PP')^G\rightarrow\PP^G$ be the factor map induced by the inclusion map $\phi_0:\PP'\rightarrow\PP$ defined so that for every $P\in\PP'$ we have $P\subseteq \phi_0(P)$. Since $X$ is not projectively isolated, there exists $Y'\in\NN_Y^F$ such that $\phi[Y']\neq X$. Again by Proposition~\ref{prop:Qmap-nbhds}~\eqref{it1-Qmap-nbhds}, there is some $\gamma\in V$ satisfying $\QQ(\gamma,\PP')=Y'$. However, then by Proposition~\ref{prop:Qmap-nbhds}~\eqref{it2-Qmap-nbhds} we get \[\QQ(\gamma,\PP)=\phi\big(\QQ(\gamma,\PP')\big)=\phi[Y']\neq X,\] which finishes the proof that $\AAA_X^\PP$ is dense since $V$ was arbitrary.\medskip

Now we just notice that there are only countably many partitions $\PP$ of $\CC$ into disjoint non-empty clopen $n$-many sets, the set of such denoted by $\mathbb{P}_n$, and therefore \[A(X)=\bigcap_{\PP\in\mathbb{P}_n} \AAA^\PP_X,\] which is dense $G_\delta$ by the Baire category theorem.

This finishes the proof of the lemma.
\end{proof}

In order to reach a contradiction, suppose now that $G$ does have the strong topological Rokhlin property, i.e. there is $\alpha\in\Act_G(\CC)$ whose conjugacy class is comeager in $\Act_G(\CC)$. Since the conjugacy class is dense it has to intersect the open set \[\UU:=\QQ^{-1}(\cdot,\PP)\big(U\big),\] where $\PP:=\{P_1,\ldots,P_n\}$ is an arbitrary partition of $\CC$ into disjoint non-empty clopen $n$-many sets. Without loss of generality, we assume that $\alpha\in\UU$. Set $X:=\QQ(\alpha,\PP)\in U$. Since by Lemma~\ref{lem:claimlemma}, $A(X)$ is dense $G_\delta$, it has to intersect the conjugacy class of $\alpha$. So there is some $\varphi\in\Homeo(\CC)$ such that \[\alpha':=\varphi \alpha \varphi^{-1}\in A(X).\] Set $\PP':=\{\varphi P_1,\ldots,\varphi P_n\}$. By definition, we have \[\QQ(\alpha',\PP')=\QQ(\alpha,\PP)=X,\] which contradicts that $\alpha'\in A(X)$. 

Applying Fact~\ref{fact:0-1-law}, we get that actually every conjugacy class is meager. This finishes the proof.
\end{proof}

\subsection{Establishing the strong topological Rokhlin property}\hfill\medskip

This subsection is devoted to the proof of the other implication of Theorem~\ref{thm:mainRokhlin}.\medskip

We assume that for each $n\geq 2$, the set of projectively isolated subshifts in $\SH_n(G)$ is dense and we prove that $G$ has the strong topological Rokhlin property.

We shall use the following proposition from \cite{BYMeTs} that is attributed to Rosendal. We are grateful to the anonymous referee for suggesting to apply this result in order to simplify the proof.
\begin{proposition}[Proposition 3.2 in \cite{BYMeTs}]\label{prop:BYMeTs}
Let $\mathbb{G}$ be a Polish group acting continuously and topologically transitively on a Polish space $X$. Then the action has a comeager orbit if and only if for every open $1_{\mathbb{G}}\in V\subseteq \mathbb{G}$ and every open $U\subseteq X$ there exists an open subset $U'\subseteq U$ such that for all pairs of further open subsets $W_1,W_2\subseteq U'$, we have $V\cdot W_1\cap W_2\neq\emptyset$.
\end{proposition}

In our case, $\mathbb{G}=\Homeo(\CC)$, $X=\Act_G(\CC)$, and the action is by conjugation which is topologically transitive, as mentioned in Introduction and proved in \cite[Proposition 1.2]{Hoch12} (we recall that the proposition is stated for $\Int^d$ but works for any countable group).\medskip

Fix an open neighborhood $\mathrm{Id}\in V\subseteq\Homeo(\CC)$ and an open set $U\subseteq \Act_G(\CC)$. Since the topology of uniform convergence on $\Homeo(\CC)$ coincides with the topology of pointwise convergence on $\mathrm{Clopen}(\CC)$ and by Lemma~\ref{lem:basicopennbhds}, without loss of generality we may assume that there are a clopen partition $\PP$ of $\CC$, a finite symmetric subset $F\subseteq G$ containing the unit, and an action $\alpha\in\Act_G(\CC)$ such that \[V=\{\phi\in\Homeo(\CC)\colon \forall P\in \PP\; (\phi[P]=P)\}\quad\text{and}\quad U=\NN_\alpha^{F,\PP}.\] Set $X':=\QQ(\alpha,\PP)\subseteq \PP^G$. Applying Proposition~\ref{prop:Qmap-nbhds}\eqref{it1-Qmap-nbhds} we get that $\QQ(\cdot,\PP)[\NN_\alpha^{F,\PP}]=\NN_{X'}^F$. By the assumption, there exists a projectively isolated subshift $X\in\NN_{X'}^F$. Using Proposition~\ref{prop:Qmap-nbhds}\eqref{it1-Qmap-nbhds} we get $\alpha'\in\NN_\alpha^{F,\PP}$ such that $\QQ(\alpha',\PP)=X$. Since $\NN_\alpha^{F,\PP}=\NN_{\alpha'}^{F,\PP}$, we may without loss of generality assume that $\alpha=\alpha'$.\medskip

By definition and using Lemma~\ref{lem:SFTnbhrds} there exists a subshift of finite type $Y\subseteq B^G$, for some non-trivial finite set $B$, a finite symmetric set $F'\supseteq F$, and a factor map $\phi:Y\to X$ such that
\begin{itemize}
	\item for every $Y'\in\NN_Y^{F'}$ we have $Y'\subseteq Y$;
	\item for every $Y'\in\NN_Y^{F'}$ we have $\phi[Y']=X$.
\end{itemize}
By Lemma~\ref{lem:standardformoffactormaps}, without loss of generality, we may assume that $\phi$ is induced by a map $\phi_0:B\to\PP$ between the alphabets. We may identify $B$ with a clopen partition $\PP'\preceq\PP$ where for $Q\in\PP'$ and $P\in\PP$ we have $Q\subseteq P$ if and only if $\phi_0(Q)=P$.

By Proposition~\ref{prop:Qmap-nbhds}\eqref{it3-Qmap-nbhds} we get $\beta\in\Act_G(\CC)$ such that $\QQ(\beta,\PP')=Y$ and $\beta\in\NN_\alpha^{F,\PP}$. Moreover, by Proposition~\ref{prop:Qmap-nbhds}\eqref{it1-Qmap-nbhds} we have $\QQ(\cdot,\PP')[\NN_\beta^{F',\PP'}]=\NN_Y^{F'}$. Let us set $U':=\NN_\beta^{F',\PP'}$. Since $\beta\in\NN_\alpha^{F,\PP}=U$, $F'\supseteq F$, and $\PP'\preceq\PP$, we get that $U'\subseteq U$.

To finish the proof, using Proposition~\ref{prop:BYMeTs}, we need to check that for all open sets $W_1,W_2\subseteq U'$ we have $V\cdot W_1\cap W_2\neq\emptyset$. Choose therefore some open $W_1,W_2\subseteq U'$. By Lemma~\ref{lem:basicopennbhds} and by refining the sets if necessary  we may assume, without loss of generality, that there are a clopen partition $\PP''\preceq\PP'$, a finite symmetric set $F''\supseteq F'$, and actions $\gamma_1,\gamma_2\in\Act_G(\CC)$ such that \[W_1=\NN_{\gamma_1}^{F'',\PP'}\quad\text{and}\quad W_2=\NN_{\gamma_2}^{F'',\PP''}.\] Since $\gamma_1,\gamma_2\in U'$ observe that $\QQ(\gamma_1,\PP)=\QQ(\gamma_2,\PP)=X$.

Consider now  $\gamma_1\times\gamma_2$, the diagonal action of $G$ on $\CC\times\CC$. Denote by $P_1:\CC\times\CC\to\CC$, resp. $P_2:\CC\times\CC\to\CC$ the projection onto the first, resp. the second coordinate and let us consider the closed $G$-invariant subspace \[W:=\{(x,y)\in\CC\times\CC\colon Q_\PP^{\gamma_1}(x)=Q_\PP^{\gamma_2}(y)\}\subseteq \CC\times\CC.\] The set $W$ is non-empty. In fact, we have $P_1[W]=P_2[W]=\CC$, therefore $Q_\PP^{\gamma_1}\circ P_1[W]=Q_\PP^{\gamma_2}\circ P_2[W]=X$. Indeed, for any $x\in\CC$ set $z:=Q_\PP^{\gamma_1}(x)\in X$ and find, since $\QQ(\gamma_2,\PP)=X$, an element $y\in\CC$ such that $Q_\PP^{\gamma_2}(y)=z$. Then $(x,y)\in W$.

Enumerate the partition $\PP''$ as $\{P_1,\ldots,P_n\}$. In what follows we shall without loss of generality assume that $W$ has no isolated points, thus it is homeomorphic to $\CC$, as otherwise we could consider the product of $W$ with $\CC$ equipped with the trivial action of $G$. We shall consider several clopen partitions of $W$ and $\CC$ respectively.

First, for every $i\leq n$ set \[U_i:=\bigcup\big\{(P_i\times P_j)\cap W\colon j\leq n\big\}\;\text{and}\;V_i:=\bigcup\big\{(P_j\times P_i)\cap W\colon j\leq n\big\}.\] 

Notice that both $\{U_i\colon i\leq n\}$ and $\{V_i\colon i\leq n\}$ are clopen partitions of $W$. Set $Z_1:=\QQ(\gamma_1,\PP'')$ and $Z_2:=\QQ(\gamma_2,\PP'')$. For each $p\in (Z_1)_{F''}$ set \[D_p^1:=\big\{(x,y)\in W\colon \forall f\in F''\;\big(\gamma_1(f^{-1})(x)\in P_i\text{ if and only if }p(f)=P_i\big)\big\},\] which is a clopen subset of $W$ and \[R_p^1:=\big\{z\in\CC\colon \forall f\in F''\;\big(\gamma_1(f^{-1})(z)=P_i\text{ if and only if }p(f)=P_i\big)\big\},\] which is a clopen subset of $\CC$. Then $\{D_p^1\colon p\in (Z_1)_{F''}\}$, resp. $\{R_p^1\colon p\in (Z_1)_{F''}\}$ are clopen partitions of $W$, resp. of $\CC$.

For each $p'\in (Z_F)_{F''}$ we define analogously the clopen sets $D_{p'}^2$ and $R_{p'}^2$ of $W$ and $\CC$ respectively.\medskip
 
Let $\psi_1:W\to\CC$ be any homeomorphism satisfying $\psi_1[D_p^1]=R_p^1$, for every $p\in (Z_1)_{F''}$. We define $\lambda_1\in\Act_G(\CC)$ by conjugating $\gamma_1\times\gamma_2\upharpoonright W$ by $\psi_1$, i.e. for $g\in G$ and $z\in\CC$ we have $\lambda_1(g)(x):=\psi_1\circ(\gamma_1\times\gamma_2(g))\circ\psi_1^{-1}(z)$. We analogously define $\psi_2:W\to\CC$ using the clopen partitions $\{D_p^2\colon p\in (Z_2)_{F''}\}$ and $\{R_p^2\colon p\in (Z_2)_{F''}\}$ of $W$ and $\CC$ respectively, and the action $\lambda_2$ which is the conjugate of $\gamma_1\times\gamma_2\upharpoonright W$ by $\psi_2$. The same arguments as in the proof of Proposition~\ref{prop:Qmap-nbhds}, which are left to the reader, show that $\lambda_1\in\NN_{\gamma_1}^{F'',\PP''}=W_1$, resp. $\lambda_2\in\NN_{\gamma_2}^{F'',\PP''}=W_2$. Since clearly $\psi_2\circ\psi_1^{-1}\lambda_1 \psi_1\circ\psi_2^{-1}=\lambda_2$, if we show that $\psi_2\circ\psi_1^{-1}\in V$ we will have verified that $V\cdot W_1\cap W_2\neq\emptyset$.\medskip

Pick any $z\in\CC$ and let $A\in\PP$ be such that $z\in A$. We show that $\psi_2\circ\psi_1^{-1}(z)\in A$. Set $(x,y):=\psi_1^{-1}(z)\in W$. Let $p\in (Z_1)_{F''}$ be such that $z\in R_p^1$ and let $P_i\in\PP''$ be such that $p(1_G)=P_i$. It follows that $P_i\subseteq A$. We have $z\in P_i$ and therefore, by the choice of $\psi_1$, we have that $x\in D_p^1$, thus also $x\in P_i\subseteq A\in\PP$. Since $Q_\PP^{\gamma_1}(x)=Q_\PP^{\gamma_2}(y)$, it follows that $y\in A$. This implies that there is $p'\in (Z_2)_{F''}$ such that $p'(1_G)=P_j\subseteq A$, for some $j\leq n$, and that $y\in D_{p'}^2$. Then again, by definition of $\psi_2$, we get $\psi_2(x,y)\in R_{p'}^2$. So in total we get \[\psi_2\circ\psi_1^{-1}(z)=\psi_2(x,y)\in R_{p'}^2\subseteq P_j\subseteq A\] which is what we wanted to prove. This finishes the proof.$\qed$
\bigskip

Theorem~\ref{thm:mainRokhlin} jointly with the recent paper \cite{PavSchmie} give yet another proof of the existence of a generic action of $\Int$ on the Cantor space. Indeed, in \cite[Theorem 3.6]{PavSchmie} they prove that isolated subshifts are dense in $\SH_\Int(n)$, for $n\geq 2$. Notice also that by Theorem~\ref{thm:mainRokhlin}, every finite group has the strong topological Rokhlin property since actually every subshift over a finite group is isolated.

\section{Free products with the strong topological Rokhlin property}\label{sect:freeproducts}
With this section we start with our applications of Theorem~\ref{thm:mainRokhlin}. Our goal is to prove Theorem~\ref{thm:intro2}. For this we need to introduce some new notions.
\begin{definition}\label{def:automaton}
	Let $A$ be a finite non-trivial set (thought of as a set of colors in this context) and let $G=\bigstar_{i\leq n} G_i$, where each for each $i\leq n$, $G_i$ is either finite or infinite cyclic. Let $S\subseteq G$ be a finite symmetric set consisting of the generator $g_i\in G_i$ and its inverse $g_i^{-1}$ if $G_i$ is infinite cyclic, and of $G_j\setminus\{1\}$ if $G_j$ is finite. A \emph{coloring automaton} is a map $\Omega: S\times A\rightarrow A$.
\end{definition}

\begin{construction}\label{con:automaton}
A coloring automaton $\Omega$ over $G$ as above and $A$ produces an element $x\in A^G$ which is uniquely defined by chosing the starting element $g\in G$ and the starting color $a\in A$, and then coloring the rest of the elements by the fixed rule given by the map $\Omega$. Moreover, the automaton keeps track on how it moves along the Cayley graph of $G$ with respect to $S$.

More formally, we proceed as follows. We define \[B:=A\times \{\leftarrow,\rightarrow,\emptyset\}^S.\] For any $b\in B$, we denote by $b_1$, resp. $b_2$ its projection to the first coordinate, which is an element of $A$, resp. its projection to the second coordinate, which is an element of $\{\leftarrow,\rightarrow,\emptyset\}^S$. Upon choosing the initial group element and color, we define an element $x\in A^G$ and an element $\tilde{x}\in B^G$ such that for all $g\in G$, $\tilde{x}(g)_1=x(g)$. \bigskip

\noindent{\bf Step 1.} Pick any $g\in G$ and $a\in A$. We set $x(g)=a$. For every $s\in S$ we set \[x(gs):=\Omega(s,a),\] and in order to keep track of how the automaton moves, we set $\tilde{x}(g)_1=a$ and for each $s\in S$ we set $\tilde{x}(g)_2(s)=\rightarrow$. Then for each $s\in S$ we set \[\tilde{x}(gs)_1:=x(gs)\text{ and }\tilde{x}(gs)_2(s^{-1})=\leftarrow.\] Additionally, for every other $s'\neq s^{-1}\in S$ we set \[\tilde{x}(gs)_2(s'):=\begin{cases}
	\emptyset & \text{if }s,s'\text{ belong to the same finite group }G_j;\\
	\rightarrow & \text{otherwise.}
\end{cases}
\]\medskip

\noindent{\bf General Step 2.} Assume that the automaton has defined $x(h)$ and $\tilde{x}(h)$, for some $h\in G$, arriving to $h$ from some $h'$ in the direction of some $s'\in S$, i.e. $h=h's'$. Then we have two cases.
\begin{enumerate}
	\item The element $s'$ does not belong to any finite group $G_j$, i.e. it is the generator or its inverse of some infinite cyclic group $G_i$. Then for every $s\neq (s')^{-1}\in S$ we set \[x(hs):=\Omega(s,x(h)),\] \[\tilde{x}(hs)_1:=x(hs)\text{ and }\tilde{x}(hs)_2(s^{-1}):=\leftarrow.\] Moreover, for every $s''\neq s^{-1}\in S$ we set \[\tilde{x}(hs)_2(s''):=\begin{cases}
		\emptyset & \text{if }s,s''\text{ belong to the same finite group }G_j;\\
		\rightarrow & \text{otherwise.}
	\end{cases}
	\]
	\item The element $s'$ belongs to some finite group $G_j$. Then for every $s\in G_j$, $x(hs)$ and $\tilde{x}(hs)$ have already been defined since $hs=h's''$ for $s''=s's\in G_j$. For $s\notin G_j$ we set as before \[x(hs):=\Omega(s,x(h)),\] \[\tilde{x}(hs)_1:=x(hs)\text{ and }\tilde{x}(hs)_2(s^{-1}):=\leftarrow.\] Moreover, for every $s''\neq s\in S$ we set \[\tilde{x}(hs)_2(s''):=\begin{cases}
		\emptyset & \text{if }s,s''\text{ belong to the same finite group }G_k;\\
		\rightarrow & \text{otherwise.}
	\end{cases}
	\]
\end{enumerate}
\end{construction}
Let us denote by $X_\Omega$, resp. $\tilde{X}_\Omega$ the closure of the subset of $A^G$ of all those elements of $A^G$, resp. closure of the subset of $B^G$ of all those elements of $B^G$, produced by the coloring automaton by the procedure above. The coordinate projection $p:B\rightarrow A$ induces a factor map $P:B^G\rightarrow A^G$.

\begin{theorem}\label{thm:automataproduceprojisolatedshfts}
	Let $A$ be a non-trivial finite set, and $G$ and $\Omega$ be a group and an automaton as in Definition~\ref{def:automaton}. Then $\tilde{X}_\Omega$ is an isolated subshift and $X_\Omega$ is a strongly projectively isolated subshift which is witnessed by the isolated factor map $P$, i.e. $P[\tilde{X}_\Omega]=X_\Omega$.
\end{theorem}

\begin{proof}
Fix the notation as in the statement. We also follow the notation from Construction~\ref{con:automaton}; in particular, we set $B:=A\times \{\leftarrow,\rightarrow,\emptyset\}^S$. The map $P$ is therefore induced by the projection map $p:B\rightarrow A$. It is clear that $P[\tilde{X}_\Omega]=X_\Omega$, so we only need to prove that $\tilde{X}_\Omega$ is isolated.

First we notice that $\tilde{X}_\Omega$ is of finite type. We define a subshift of finite type $X\subseteq B^G$ and show that $X=\tilde{X}_\Omega$. We set the defining window to be $F:=S\cup\{1\}$. A pattern $p\in B^F$ is allowed if it is produced by the automaton $\Omega$, i.e. if and only if $p\in (\tilde{X}_\Omega)_F$. Clearly, $\tilde{X}_\Omega\subseteq X$. We show the converse. Pick $x\in X$. We distinguish two cases.\smallskip

\noindent{\bf Case 1.} We have that there is $g\in G$ such that $x(g)_2(s)=\rightarrow$ for all $s\in S$. Then it is straightforward to check that the rules imposed by the allowed patterns of $X$ force $x$ to correspond to an element produced by $\Omega$ when the starting vertex is $g\in G$ and the starting color is $x(g)_1$. It follows by definition that $x\in \tilde{X}_\Omega$.\smallskip

\noindent{\bf Case 2.} For every $g\in G$ there is one, and consequently by the rules imposed by allowed patterns \emph{exactly one}, $s\in S$ such that $x(g)_2(s)=\leftarrow$. Pick an arbitrary $g\in G$ and set $g_1:=g$. By the assumption there is unique $s\in S$ such that $x(g)_2(s)=\leftarrow$. Set $g_2:=gs$. Assuming that $g_{n-1}$, for $n\in\Nat$ has been defined, we set $g_n:=g_{n-1}s$, where $s$ is the unique element of $S$ such that $x(g_{n-1})_2(s)=\leftarrow$. This defines an infinite path $(g_n)_{n\in\Nat}$. This path is unique up to finite initial segment. Indeed, choosing a different starting element $h\in G$ there is a unique path $h_1=h,\ldots,h_n=g$ such that $h_i^{-1}h_i\in S$, for $1<i\leq n$, and such that either for each $i<n$ we have $x(h_i)_2(h_{i+1}^{-1}h_i)=\leftarrow$, or for each $i<n$ we have $x(h_i)_2(h_{i+1}^{-1}h_i)=\rightarrow$. Notice that such a path between $h$ and $g$ is unique although there may be more than one path $h'_1=h,\ldots,h'_m=g$ such that $(h'_i)^{-1}h'_i\in S$, for $1<i\leq m$.

Now for every $n\in\Nat$ we define an element $x_n\in \tilde{X}_\Omega$, which is an element of $B^G$ produced by $\Omega$ with the starting element $g_n$ and the starting color $x(g_n)_1$. It is again straightforward to check that $\lim_{n\to\infty} x_n=x$, so we get that $x\in\tilde{X}_\Omega$ as desired.\medskip

It remains to show that $\tilde{X}_\Omega=X$ is isolated. We claim that $\NN_X^F=\{X\}$. Indeed, this follows from the facts that
\begin{itemize}
	\item for every $X'\in\NN_X^F$ we have $X'\subseteq X$;
	\item every $X'\subseteq X$ that contains the pattern $p_a\in B^F$, which is defined as $x_a\upharpoonright F$, where $x_a\in\tilde{X}_\Omega$ is the element of $B^G$ produced by $\Omega$ with the starting element $1_G$ and the starting color $a$, must be equal to $X$. Indeed, $X=\tilde{X}_\Omega$ is by definition the closure of the set of configurations produced by $\Omega$, which is equal to the smallest subshift containing configurations produced by $\Omega$ with the starting element $1_G$ and all possible colors $a\in A$.
\end{itemize}
\end{proof}

\begin{lemma}\label{lem:nbhdisomorphism}
	Let $A$ and $B$ be non-trivial finite sets, $G$ be a countable group, and $X\subseteq A^G$ and $Y\subseteq B^G$ be two subshifts of finite type such that there is an isomorphism $\phi:X\rightarrow Y$. Then $\phi$ canonically induces a bijection $\Phi$ between an open neighborhood $\NN_X$ of $X$ and $\NN_Y$ so that for every $Z\in\NN_X$, $Z$ is isomorphic to $\Phi(Z)$.
\end{lemma}
\begin{proof}
	Let us fix $A$, $B$, $G$, $X$, $Y$, and $\phi:X\rightarrow Y$ as in the statement. We may suppose that there is a finite set $F\subseteq G$ such that $F$ is the defining window for $X$ and $\phi$ is determined by a map $f:X_F\rightarrow B$. Set $\NN_X:=\NN_X^F$ and for $Z\in\NN_X$ set \[\Phi(Z):=\phi[Z].\] 
	We leave to the reader the straightforward verification that for every $Z\in\NN_X$, $\Phi(Z)$ is isomorphic to $Z$, and that $\NN_Y:=\Phi[\NN_X]$ is an open neighborhood of $Y$.
\end{proof}

\begin{definition}
	Let $G$ and $H$ be countable groups, $A$ be a non-trivial finite set, and $X\subseteq A^G$, $Y\subseteq A^H$ subshifts. The \emph{free product} $X\ast Y$ of $X$ and $Y$ is a subshift of $A^{G\ast H}$ defined as \[\{x\in A^{G\ast H}\colon \forall g\in G\ast H\; (gx\upharpoonright G\in X\wedge gx\upharpoonright H\in Y)\}.\] Equivalently, if $P_X$, resp. $P_Y$ are the set of forbidden patterns of $X$, resp. of $Y$, then $X\ast Y$ is defined as the subshift of $A^{G\ast H}$ whose set of forbidden patterns is $P_X\cup P_Y$.
	
	In particular, if $X$ and $Y$ are subshifts of finite type, then so is $X\ast Y$.
\end{definition}

\begin{definition}
	Let $G$ be a countable group, $A$ and $B$ be non-trivial finite sets, and let $X\subseteq A^G$ and $Y\subseteq B^G$ be subshifts. Say that a continuous $G$-equivariant map $\phi:Y\rightarrow X$ is \emph{basic} if it is induced by a map $\phi_0:B\rightarrow A$ between the respective alphabets.
\end{definition}
\begin{lemma}\label{lem:freeprodofshifts}
	Let $G$ and $H$ be countable groups, $A$ be a non-trivial finite set, and $X\subseteq A^{G\ast H}$ be a subshift of finite type. Then there are a non-trivial finite set $B$ and a basic isomorphism $\phi$ between a free product $V\ast W$ and $X$, where $V\subseteq B^G$ and $W\subseteq B^H$ are subshifts of finite type. Moreover, if $\NN$ is an open neighborhood of $X$, then $\phi$ may be chosen so that it induces a bijection $\Phi$ between the open neighborhoods $\NN':=\NN_{V\ast W}^{\{1\}}$ of $V\ast W$ and $\Phi[\NN']\subseteq \NN$ of $X$ such that for all $Y\in\NN'$, $Y$ and $\Phi(Y)$ are isomorphic.
\end{lemma}
\begin{proof}
	Fix $G$ $H$, $A$, and $X$ as in the statement. By Proposition~\ref{prop:SFTprop}, there are finite sets $S_G\subseteq G$ and $S_H\subseteq H$, a non-trivial finite set $B$, and a subshift $X'\subseteq B^{G\ast H}$ isomorphic to $X$ so that for $x\in B^{G\ast H}$ we have $x\in X'$ if and only if for every $g\in G\ast H$ and for each $s\in S_G\cup S_H$ \[gx\upharpoonright \{1,s\}\text{ is an allowed pattern}.\] The set of allowed patterns defined on the subsets $\{1,s\}$, where $s\in S_G$, resp. $s\in S_H$, will be denoted by $P_G$, resp. by $P_H$.
	
	We define $V\subseteq B^G$ to be the subshift of finite type, where for $x\in B^G$ we have $x\in V$ if and only if for every $g\in G$ and $s\in S_G$ \[gx\upharpoonright \{1,s\}\in P_G.\] We define a subshift $W\subseteq B^H$ of finite type analogously, with $S_H$ and $P_H$ instead of $S_G$ and $P_G$. We leave to the reader the straightforward verification that $X'=V\ast W$, and thus $X$ is isomorphic to $V\ast W$.
	
	For the `moreover' part, assume without loss of generality, that $\NN=\NN_X^F$, for some finite $F\subseteq G\ast H$ containing the unit. Then, using Lemma~\ref{lem:SFTfactorofSFT} the map $X_F\to A$ defined by $p\in X_F\to p(1)$ induces an isomorphism $\psi$ between a subshift $Z\subseteq (X_F)^{G\ast H}$ and $X$ such that $\psi[\NN_Z^{\{1\}}]=\NN$. Therefore, without loss of generality, we may assume that $\NN=\NN_X^{\{1\}}$. Now since $\phi$ is induced by a map $B\to A$ (see the remark following the statement of Proposition~\ref{prop:SFTprop}), the conclusion follows from Lemma~\ref{lem:nbhdisomorphism}.
\end{proof}

\begin{definition}
Let $G$ and $H$ be countable groups and let $A,B,C$ be non-trivial finite sets. Let $\phi_0:B\rightarrow A$ and $\psi_0:C\rightarrow A$ be maps and denote by $B\times_{\phi,\psi} C$ the restricted direct product $\{(b,c)\in B\times C\colon \phi_0(b)=\psi_0(c)\}$. Denote also by $p_B$, resp. $p_C$ the projection from $B\times C$ on the first, resp. the second coordinate, and by $P_B$, resp. $P_C$ the corresponding induced maps from $(B\times C)^G$ onto $B^G$, resp. from $(B\times C)^H$ onto $C^H$. Suppose that $X\subseteq B^G$ and $Y\subseteq C^H$ are subshifts. Then the \emph{restricted free product} $X\ast_{\phi,\psi} Y$ of $X$ and $Y$ with respect to the maps $\phi_0$ and $\psi_0$ is a subshift of $(B\times_{\phi,\psi}C)^{G\ast H}$ defined as \[\{x\in (B\times_{\phi,\psi}C)^{G\ast H}\colon \forall g\in G\ast H\;\big(P_B(gx\upharpoonright G)\in X\wedge P_C(gx\upharpoonright H)\in Y\big)\}.\]
\end{definition}

\begin{lemma}\label{lem:isooffreeprodofSFT}
Let $G$ and $H$ be countable groups and let $A,B,C$ be non-trivial finite sets. Let $\phi_0:B\rightarrow A$ and $\psi_0:C\rightarrow A$ be maps which induce maps $\phi:B^G\rightarrow A^G$, resp. $\psi:C^H\rightarrow A^H$. Suppose that $X\subseteq A^G$ and $X'\subseteq B^G$, resp. $Y\subseteq A^H$ and $Y'\subseteq C^H$ are subshifts such that $\phi$, resp. $\psi$ induces an isomorphism between $X'$ and $X$, resp. between $Y'$ and $Y$. Then the map $(\phi_0,\psi_0):B\times_{\phi,\psi}C\rightarrow A$ induces an isomorphism between $X'\ast_{\phi,\psi} Y'$ and $X\ast Y$.
\end{lemma}
\begin{proof}
Fix the notation as in the statement. The canonical map $(\phi_0,\psi_0):B\times_{\phi,\psi}C\rightarrow A$ induces a map $\eta:(B\times_{\phi,\psi}C)^{G\ast H}\rightarrow A^{G\ast H}$. We need to show that $\eta[X'\ast_{\phi,\psi} Y']=X\ast Y$ and $\eta\upharpoonright X'\ast_{\phi,\psi} Y'$ is one-to-one.

First we check that $\eta[X'\ast_{\phi,\psi} Y']\subseteq X\ast Y$. Pick $x\in X'\ast_{\phi,\psi} Y'$ and let us show that $\eta(x)\in X\ast Y$. We need to verify that for every $g\in G\ast H$ we have $\eta(gx)\upharpoonright G\in X$ and $\eta(gx)\upharpoonright H\in Y$. Fix $g\in G\ast H$. By definition, we have $P_B(gx\upharpoonright G)\in X'$ and $P_C(gx\upharpoonright H)\in Y'$. Notice however that we have \[\eta(gx)\upharpoonright G=\phi\big(P_B(gx\upharpoonright G)\big)\in\phi[X']=X.\] The case of $\eta(gx)\upharpoonright H$ is analogous.

Next we show injectivity and surjectivity of the map $\eta:X'\ast_{\phi,\psi} Y'\rightarrow X\ast Y$. Let us start with the injectivity. Pick $x\neq y\in X'\ast_{\phi,\psi} Y'$. Without loss of generality we can assume that $x(1)\neq y(1)$. It follows that either $P_B(x\upharpoonright G)\neq P_B(y\upharpoonright G)$ or $P_C(x\upharpoonright H)\neq P_C(x\upharpoonright H)$. Assume the former, the latter is treated analogously. Then \[\eta(x)\upharpoonright G=\phi\big(P_B(x\upharpoonright G)\big)\neq \phi\big( P_B(y\upharpoonright G)\big)=\eta(y)\upharpoonright G,\] since $\phi$ is injective. However, the inequality $\eta(x)\upharpoonright G\neq\eta(y)\upharpoonright G$ implies the inequality $\eta(x)
\neq \eta(y)$.

Finally, we show the surjectivity. First we notice that since $\phi$ and $\psi$ are isomorphisms, the maps $\phi^{-1}:X\rightarrow X'$, resp. $\psi^{-1}: Y\rightarrow Y'$ are also isomorphisms induced, by the Curtis-Hedlund-Lyndon theorem, by some maps $f_X: X_F\rightarrow B$, resp. $f_Y: Y_E\rightarrow C$, where $F\subseteq G$ and $E\subseteq H$ are finite sets. Let $y\in X\ast Y$. Let us define $x\in X'\ast_{\phi,\psi} Y'$ as follows. For $g\in G\ast H$ we set $x(g):=(b,c)$, where \[b:=f_X(gy\upharpoonright F),\; c:=f_Y(gy\upharpoonright E).\] Since $\phi_0\big(f_X(gy\upharpoonright F)\big)=y(g)=\psi_0\big(f_Y(gy\upharpoonright E)\big)$, we have that $(b,c)\in B\times_{\phi,\psi} C$. In order to check that $x\in X'\ast_{\phi,\psi} Y'$ we need to verify that for every $g\in G\ast H$ we have $P_B(gx\upharpoonright G)\in X'$ and $P_C(gx\upharpoonright H)\in Y'$. We verify the former, the latter is analogous, and without loss of generality, we assume that $g=1$. However, then \[P_B(x\upharpoonright G)=\phi^{-1}(y\upharpoonright G)\in \phi^{-1}[X]=X'.\]
It remains to check that $\eta(x)=y$, which is straightforward and left to the reader.
\end{proof}

\begin{proposition}\label{prop:freeprodsautomata}
	Let $G$ and $H$ be countable groups of the form allowed in Definition~\ref{def:automaton}. Suppose that for any subshift of finite type $X$, resp. $Y$, over $G$, resp. over $H$, and every open neighborhood $\NN_X$ of $X$, resp. $\NN_Y$ of $Y$, there is a subshift $X'\in\NN_X$, resp. $Y'\in\NN_Y$ isomorphic via a basic isomorphism to a subshift produced by a coloring automaton. Then the same holds true for every open neighborhood of any subshift of finite type over $G\ast H$.
\end{proposition}
\begin{proof}
Fix $G$, $H$, and let $V$ be a subshift of finite type over $G\ast H$ and $\NN$ be an open neighborhood of $V$. By Lemma~\ref{lem:freeprodofshifts}, there exist a finite non-trivial set $A$, subshifts of finite type $X\subseteq A^G$ and $Y\subseteq A^H$, and a basic isomorphism $\phi:X\ast Y\to V$ that induces a bijection between an open neighborhood $\NN_{X\ast Y}^{\{1\}}$ and open neighborhood $\NN'\subseteq \NN$ of $V$. Consequently, without loss of generality, we may assume that $V=X\ast Y$ and $\NN=\NN_{X\ast Y}^{\{1\}}$, i.e. for $Z\subseteq A^{G\ast H}$ we have $Z\in\NN$ if and only if $Z$ is \emph{fully colored}, i.e. for every $a\in A$ there is $z\in Z$ with $z(1)=a$. By the assumption there are $X'\subseteq X$ and $Y'\subseteq Y$ that are fully colored, i.e. for every $a\in A$ there are $x\in X'$, resp. $y\in Y'$ satisfying $x(1)=a=y(1)$, and moreover $X'$ and $Y'$ are isomorphic to subshifts $X''\subseteq B^G$, resp. $Y''\subseteq C^H$, for some non-trivial finite sets $B$ and $C$ that are produced by some coloring automata $\Omega_X$ and $\Omega_Y$ defined on finite symmetric generating sets $S_X\subseteq G$ and $S_Y\subseteq H$. Moreover, the isomorphisms $\phi:X''\rightarrow X'$, resp. $\psi:Y''\rightarrow Y'$ are basic, i,e. induced by finite (surjective) maps $\phi_0:B\rightarrow A$, resp. $\psi_0: C\rightarrow A$.

We define a new coloring automaton on $G\ast H$ with the set of colors $B\times_{\phi,\psi} C$ and with respect to the finite symmetric generating set $S:=S_X\cup S_Y$. For $(b,c)\in B\times_{\phi,\psi} C$ and $s\in S$ we set \[\Omega(s,(b,c)):=\begin{cases}
	(\Omega_X(s,b),c') & \text{if }s\in S_X\text{ so that }(\Omega_X(s,b),c')\in B\times_{\phi,\psi} C;\\
	(b',\Omega_Y(s,c)) & \text{if }s\in S_Y\text{ so that }	(b',\Omega_Y(s,c))\in B\times_{\phi,\psi} C.
\end{cases}\] In other words, if $s$ belongs to, say, $S_X$, $\Omega(s,(b,c))$ is $(b',c')$, where $b'$ is produced by the automaton $\Omega_X$ and $c'$ is arbitrary so that $(b',c')\in B\times_{\phi,\psi} C$.

Denote by $Z\subseteq (B\times_{\phi,\psi} C)^{G\ast H}$ the subshift produced by the coloring automaton $\Omega$. Clearly we have $Z\subseteq X''\ast_{\phi,\psi} Y''$. By Lemma~\ref{lem:isooffreeprodofSFT}, we have an isomorphism $\eta: X''\ast_{\phi,\psi} Y''\rightarrow X'\ast Y'$. Therefore $Z':=\eta[Z]\subseteq X'\ast Y'\subseteq X\ast Y$. Since $Z$ is fully colored, i.e. for every $(b,c)\in B\times_{\phi,\psi} C$ there is $z\in Z$ such that $z(1)=(b,c)$, also $Z'$ is fully colored. Therefore $Z\in\NN$.
\end{proof}

\begin{corollary}\label{cor:stabilityunderfreeprod-automata}
	Let $G$ and $H$ be countable groups such that for each $F\in\{G,H\}$ and every non-trivial finite set $A$ the set of those subshifts that are isomorphic to subshifts produced by coloring automata is dense in $\SH_F(A)$. Then the same is true for $F=G\ast H$ and any non-trivial finite set $A$.
	
	In particular, the set of strongly projectively isolated subshifts is dense in $\SH_F(A)$, for every non-trivial finite set $A$, and $F$ has the strong topological Rokhlin property.
\end{corollary}
\begin{proof}
Fix the groups $G$ and $H$ and a non-trivial finite set $A$. Let $\NN\subseteq \SH_{G\ast H}(A)$ be an open neighborhood. By Lemma~\ref{lem:SFTnbhrds}, we may suppose that $\NN$ is an open neighborhood of some subshift of finite type $X$. Then we apply Proposition~\ref{prop:freeprodsautomata}. The `in particular' part follows then from Theorems~\ref{thm:automataproduceprojisolatedshfts} and~\ref{thm:mainRokhlin}.
\end{proof}

We have already remarked that in \cite[Theorem 3.6]{PavSchmie} they prove that isolated subshifts are dense in $\SH_\Int(n)$, for $n\geq 2$. In fact, they prove that subshifts that are isomorphic to subshifts of finite type given by Rauzy graphs with no middle cycles property are isolated and dense. Here a Rauzy graph $(V,E)$ (in our previous notation from Definition~\ref{def:Rauzy}, an $\{1\}$-Rauzy graph) has \emph{no middle cycles property} if it has no cycle that has both incoming and outgoing edge from outside of the cycle - we refer to \cite[Definition 3.1 and Lemma 3.3]{PavSchmie} for more details.
\begin{theorem}\label{thm:STRPfreeproducts}
	Let $(G_i)_{i\leq n}$ be a finite sequence of groups where each of them is either finite or cyclic. Then $\bigstar_{i\leq n} G_i$ has the strong topological Rokhlin property.
	
	In fact, the set of strongly projectively isolated subshifts is dense in the spaces of subshifts over $\bigstar_{i\leq n} G_i$.
\end{theorem}
\begin{proof}
In order to apply Corollary~\ref{cor:stabilityunderfreeprod-automata} it suffices to show that for any group $G$ that is either finite or infinite cyclic and every non-trivial finite set $A$ the set of those subshifts produced by coloring automata is dense in $\SH_G(A)$.\medskip

\noindent{\bf Case 1.} $G$ is finite. Fix a non-trivial finite set $A$ and a subshift $X\subseteq A^G$. Set $B:=X_G$ and define a subshift $Y\subseteq B^G$ isomorphic to $X$ via a map $\phi:Y\rightarrow X$ which induced by the map $\phi_0:B=X_G\rightarrow A$ defined by $p\in X_G\to p(1_G)$. Now we define an automaton $\Omega$ with respect to $S=G\setminus\{1_G\}$ and $B$ as the set of colors as follows. For $b\in B=X_G$ and $g\in G\setminus\{1_G\}$ we set \[\Omega(g,b):=g^{-1}b.\]
It is straighforward that the subshift produced by $\Omega$ is equal to $Y$.\medskip

\noindent{\bf Case 2.} $G$ is infinite cyclic.

Since $G$ is isomorphic to $\Int$, we shall use the additive notation for group operations. We fix a non-trivial finite set $A$ and a subshift of finite type $X\subseteq A^\Int$ and some nighborhood $\NN$ od $X$. By \cite[Theorem 3.6]{PavSchmie}, $\NN$ contains an NMC (no middle cycle) subshift $Y$, where $Y$ has NMC if there exists $n\in\Nat$ such that the Rauzy graph on $Y_{[0,n]}$ (i.e. the graph where $Y_{[0,n]}$ is the set of vertices and the set of edges is defined in the obvious way as in Proposition~\ref{prop:SFTprop}) has no middle cycles property. By increasing $n$ we may moreover assume that the graph, which we shall denote $(V,E)$, has no vertex that has both more than one incoming edges and more than one outgoing edges. We let $P:V\rightarrow\{0,1\}$ to be the characteristic function of the set of those vertices of $V$ that belong to a simple cycle.

We shall now define the automaton as follows, with respect to $S=\{-1,1\}$ and $V$ as the set of colors. Pick $v\in V$. We define \[\Omega(1,v):=\begin{cases}
	w & \text{if }P(v)=0\text{ and }w\text{ is arbitrary such that }(v,w)\in E;\\
	w & \text{if }P(v)=1,\; P(w)=1\text{ and }(v,w)\in E.
\end{cases}\]
In words, we set $\Omega(1,v)$ to be the unique successor of $v$ on a simple cycle provided that $v$ lies on a simple cycle, otherwise we set $\Omega(1,v)$ to be an arbitrary successor ov $v$ in the graph. We define $\Omega(-1,v)$ symmetrically.

It is clear that the subshift generated by $\Omega$ is a subshift $Y'$ of $Y_{[0,n]}^\Int$ isomorphic to $Y$ via a map $\phi:Y'\rightarrow Y$ induced by $\phi_0:Y_{[0,n]}\to A$ defined by  $p\in Y_{[0,n]}\to p(0)$.
\end{proof}

\section{Groups without the strong topological Rokhlin property}\label{sect:noSTRP}
In this section we initiate the production of negative examples, i.e. we find many new examples of groups without the strong topological Rokhlin property. We prove here Theorem~\ref{thm:intro3}. We start with one more characterization of the strong topological Rokhlin property that is useful especially for proving the negative results.

When $X$ is a subshift and $\AAA$ is a collection of closed subsets of $X$, not necessarily a partition, we say that $X$ is \emph{$\AAA$-minimal} if there is no non-empty proper subshift $Y\subseteq X$ that intersects every element from $\AAA$.

\begin{theorem}\label{thm:2ndcharacterizationofSTRP}
A countable group $G$ has the strong topological Rokhlin property if and only if for every non-trivial finite set $A$, any sofic subshift $X\subseteq A^G$, which is a factor of a subshift of finite type $Z\subseteq B^G$, for some non-trivial finite $B$, via a factor map $\phi: Z\rightarrow X$, and any open neighborhood $\NN$ of $X$ in $\SH_G(A)$, there are 
\begin{itemize}
	\item a sofic subshift $Y\subseteq Z$ satisfying $\phi[Y]\in\NN$,
	\item a subshift of finite type $V\subseteq C^G$, for some non-trivial finite set $C$, factoring on $Y$ via some $\psi: V\rightarrow Y$,
	\item and a clopen partition $\PP$ of $V$  such that $\phi[Y]$ is $\AAA$-minimal, where $\AAA=\phi\circ\psi(\PP)$.
\end{itemize}
  
\end{theorem}
\begin{proof}
Fix a countable group $G$.\medskip

Suppose first that there are a non-trivial finite set $A$, a sofic subshift $X\subseteq A^G$, which is a factor of a subshift of finite type $Z\subseteq B^G$, for some non-trivial finite $B$, via a factor map $\phi: Z\rightarrow X$, and an open neighborhood $\NN$ of $X$ in $\SH_G(A)$ such that no sofic subshift $Y\subseteq Z$ satisfying $\phi[Y]\in\NN$ is $\AAA$-minimal for any $\AAA:=\phi\circ\psi[\PP]$, where $V\subseteq C^G$, for some non-trivial finite set $C$, factors onto $Y$ via $\psi:V\rightarrow Y$ and $\PP$ is a clopen partition of $V$.

The open neighborhood of $X$ is without loss of generality of the form $\NN_X^F$ for some finite set $F\subseteq G$. Each pattern $p\in X_F$ determines a basic clopen set $C_p\subseteq X$ so that $\RR:=\{C_p\colon p\in X_F\}$ is a clopen partition of $X$. The preimage of $\RR$ via $\phi$ is a clopen partition $\PP'$ of $Z$. We may refine $\PP'$ to a clopen partition $\PP''$ that consists of basic clopen sets $C_p(Z)$ indexed by patterns $p\in Z_E$, where $E\subseteq G$ is a finite set, and so that for every $Z'\in\NN_Z^E$ we have that $Z'\subseteq Z$ - the latter claim by Lemma~\ref{lem:SFTnbhrds}. We now claim that $\NN_Z^E$ contains no projectively isolated subshift. If we show it, applying Proposition~\ref{prop:noRokhlin}, we get that $G$ does not have the strong topological Rokhlin property.

Suppose on the contrary that $Z'\in\NN_Z^E$ is projectively isolated. In particular, it is a sofic subshift of $Z$. Let $V\subseteq C^G$ be a subshift, for some finite non-trivial set $C$, such that there are an open neighborhood $\NN$ of $V$ in $\SH_G(C)$ and a map $\psi$ defined on every subshift from $\NN$ so that for every $V'\in\NN$ we have $\psi[V']=Z'$. We may suppose that $V$ is of finite type and that $\NN=\NN_V^J$, for some finite set $J\subseteq G$, so that for every $V'\in\NN_V^J$ we have $V'\subseteq V$, again by applying Lemma~\ref{lem:SFTnbhrds}. Set \[Y:=\phi\circ\psi[V].\] Since $\psi[V]=Z'\in \NN_Z^E$ we get $\psi[V]\subseteq Z$, and thus $Y=\phi[Z']\subseteq\phi[Z]=X$. Second, since $\psi[V]\in\NN_Z^E$ we get that $\psi[V]$ intersects every clopen subset from $\PP'$, and thus $Y=\phi\circ\psi[V]$ intersects every clopen subset of $\RR$, so $Y\in\NN_X^F$. Set now \[\AAA:=\{\phi\circ\psi\big(C_p(V)\big)\colon p\in V_J\}.\] We have $Y=\bigcup \AAA$ and by the assumption, $Y$ is not $\AAA$-minimal. Therefore there is a proper subshift $Y'\subseteq Y$ that non-trivially intersects every set from $\AAA$. However then \[V':=(\phi\circ\psi)^{-1}(Y')\] is a proper subshift of $V$ that non-trivially intersects every set $C_p(V)$, for $p\in V_J$. In other words, $V'\in\NN_V^J$. But then by the assumption on the projective isolation of $Z'$ we get $\psi[V']=Z'$. Therefore \[Y'=\phi\circ\psi[V']=\phi[Z']=\phi\circ\psi[V]=Y,\] which is a contradiction as the inclusion $Y'\subseteq Y$ is strict.\bigskip

Now we prove the converse. Using Theorem~\ref{thm:mainRokhlin}, it suffices to show that projectively isolated subshifts are dense in $\SH_G(A)$, for any non-trivial finite set $A$. Fix such $A$ and an open set $\NN\subseteq\SH_G(A)$. We may suppose that $\NN$ is of the form $\NN_X^F$, where $X\subseteq A^G$ is a subshift of finite type, $F\subseteq G$ is a finite set, and moreover for every $X'\in\NN_X^F$ we have $X'\subseteq X$. Now we use the assumption. Notice that since $X$ is of finite type, we have $Z=X$ and $\phi=\mathrm{id}$ in the notation of the statement. By the assumption, there is a sofic subshift $Y\in\NN_X^F$ that is $\AAA$-minimal, where $\AAA$ is an image of a clopen partition $\PP$ of some subshift of finite type $V\subseteq C^G$, for some non-trivial finite set $C$, that factors onto $Y$. We claim that $Y$ is projectively isolated. Denote by $\psi$ the factor map from $V$ onto $Y$. We may refine the clopen partition $\PP$ to a clopen partition \[\PP':=\{C_p(V)\colon p\in V_E\},\] where $E\subseteq G$ is a finite set. We can assume that $E$ is big enough so that for every $V'\in\NN_V^E$ we have $V'\subseteq V$. The image of $\PP'$ is a family of closed sets $\AAA'$ that refines $\AAA$, i.e. for every $R'\in\AAA'$ there is $R\in\AAA$ such that $R'\subseteq R$. It is clear that since $Y$ is $\AAA$-minimal it is also $\AAA'$-minimal. In order to show that $Y$ is projectively isolated it is enough to show that for every $V'\in\NN_V^E$ we have $\phi[V']=Y$. Pick $V'\in\NN_V^E$. First of all notice that since $V'\subseteq V$, $\psi$ is indeed defined on $V'$. Second, since $V'\in\NN_V^E$, by definition $V'\cap C_p(V)\neq\emptyset$ for every $p\in V_E$. Thus $\phi[V']\cap\phi[C_p(V)]\neq\emptyset$, for every $p\in V_E$, so $V'$ intersects non-trivially every element of $\AAA'$. However, since $Y$ is $\AAA'$-minimal, it follows that $\phi[V']=Y$.
\end{proof}

\subsection{Finitely generated groups without the STRP}
We start with results that concern finitely generated groups. As we shall see it is much harder to disprove the strong topological Rokhlin property for finitely generated groups. Before stating the main result, we ask the reader to pay attention to the following facts that play an important role in the proof below. If $A$ is some non-trivial finite set, $G$ is a finitely generated recursively presented group, $N$ is its finitely generated normal subgroup, $X\subseteq A^{G/N}$ is a subshift, and $X'\subseteq A^G$ is the corresponding subshift over $G$ where $N$ acts trivially, then 
\begin{itemize}
	\item $X$ is effective if and only if $X'$ is (this follows from Lemma~\ref{lem:effectivesubshiifts});
	\item if $X$ is of finite type, resp. sofic, then so is $X'$ (this will be proved below);
	\item if $X'$ is sofic, then $X$ is effective, but not necessarily sofic (this is the point of the so-called simulation theorems, we refer to \cite{Bar19} for more details).
\end{itemize}
These facts are behind the reason why some of the subshifts in the theorem below are chosen to be effective, resp. sofic, resp. of finite type.
\begin{theorem}\label{thm:disjointunionpropersubshifts}
Let $G$ be a countable group. Suppose that for some non-trivial finite set $A$ there is a sofic subshit $X\subseteq A^G$, being a factor of a subshift of finite type $Z\subseteq B^G$, for some non-trivial finite set $B$, via a map $\phi: Z\rightarrow X$, such that for every effective subshift $Y\subseteq Z$, $\phi[Y]$ is a disjoint union of infinitely many proper subshifts. 

Suppose also that there is a short exact sequence of groups \[1\to N\to H\to G\to 1,\] where $H$ is finitely generated and recursively presented, and $N$ is finitely generated. Then $H$, and in particular also $G$, does not have the strong topological Rokhlin property.
\end{theorem}
\begin{proof}
Fix the groups $G$, $N$, and $H$, the non-trivial finite sets $A$ and $B$, and the subshifts $X\subseteq A^G$, $Z\subseteq B^G$, and the factor map $\phi: Z\rightarrow X$ - as in the statement. Notice first that since $H$ is finitely generated and recursively presented, and $N$ is finitely generated, it follows that $G$ is also finitely generated and recursively presented, so the use of effective subshifts over $G$ is in accordance with Definition~\ref{def:effectiveshubshift}.

 Let $V\subseteq A^H$ be the subshift \[\{v\in A^H\colon \forall g\in H\;\forall h\in N\; \big(v(g)=v(gh)\big)\},\] and $W\subseteq B^H$ be the subshift \[\{w\in B^H\colon \forall g\in H\;\forall h\in N\; \big(w(g)=w(gh)\big)\}.\]
Clearly, there are a 1-1 correspondence \[Y\subseteq A^G\to Y'\subseteq V\] between subshifts of $A^G$ and subshifts of $V$, and a 1-1 correspondence \[Y\subseteq B^H\to Y'\subseteq W.\] Let therefore $X'\subseteq V$ be the subshift of $V$, and so of $A^H$, that corresponds to $X$, and let $Z'\subseteq W$ be the subshift of $W$, and so of $B^H$, that corresponds to $Z$. Assuming without loss of generality, by applying Lemma~\ref{lem:factoringonshift}, that $\phi: Z\rightarrow X$ is induced by a map $f:B\rightarrow A$ between the alphabets, the same map $f$ also induces a factor map $\phi': Z'\rightarrow X'$. We claim that $Z'$ is of finite type. To see this, let $S\subseteq N$ be a finite symmetric generating set of $N$. Let $\FF$ be the finite set of forbidden patterns for $Z$. We may suppose that they are all defined on a finite set $T\subseteq G$. For each $t\in T$ choose one $t'\in H$ such that $Q(t')=t$, where $Q:H\rightarrow G$ is the quotient map. For each pattern $p\in\PP$ define a pattern $p':T'\rightarrow A$ by \[p'(t')=p(t)\;\;\text{ for }t'\in T',\] and denote by $\FF'$ the set of such patterns. Moreover, for every $s\in S$ and $a\neq b\in A$ define a pattern $p_{s,a,b}:\{1,s\}\to A$ by \[p_{s,a,b}(1)=a\text{ and }p_{s,a,b}(s)=b,\] and denote by $\FF''$ the set of patterns $\{p_{s,a,b}\colon s\in S, a\neq b\in A\}$. It is now straightforward to verify and left to the reader that the set $\FF'\cup\FF''$ of forbidden patterns defines the subshift $Z'$. It also follows that $X'$ as a factor of $Z'$ is sofic. Now we show that for no sofic subshift $Y'\subseteq Z'$, $\phi'[Y']$ is $\phi'[\AAA]$-minimal for $\AAA$ being an image of clopen partition of some subshift of finite type factoring onto $Y'$. Once this claim is shown, we apply Theorem~\ref{thm:2ndcharacterizationofSTRP} to conclude that $H$ does not have the strong topological Rokhlin property.

Let us therefore prove the claim. Suppose on the contrary that $Y'\subseteq Z'$ is a sofic subshift which is a factor of a subshift of finite type $V\subseteq C^H$, for some non-trivial finite set $C$, via some $\psi: V\rightarrow Y'$, and that there is a clopen partition $\PP$ of $V$ such that for $\AAA:=\psi[\PP]$ we have that $\phi'[Y']$ is $\phi'[\AAA]$-minimal. As in the proof of Theorem~\ref{thm:2ndcharacterizationofSTRP}, we may suppose that $\PP$ is of the form $\{C_p(V)\colon p\in V_F\}$, for some finite set $F\subseteq H$, and moreover for every $V'\in\NN_V^F$ we have $V'\subseteq V$. Suppose that $Y$, the subshift of $Z$ corresponding to $Y'$, is effective. We show how to finish the proof then. For each $D'\in\phi'[\AAA]$, which is a closed subset of $\phi'[Y']$, denote by $D$ the corresponding closed subset of $\phi[Y]$.

Since by the assumption $\phi[Y]$ is a disjoint union of infinitely many proper subshifts, for each $D$ corresponding to $D'\in\phi'[\AAA]$ we can find a proper subshift $Y_D\subseteq \phi[Y]$ with $Y_D\cap D\neq\emptyset$, and moreover that for $D_1\neq D_2$ where $D'_1\neq D'_2\in\phi'[\AAA]$ either $Y_{D_1}=Y_{D_2}$ or $Y_{D_1}\cap Y_{D_2}=\emptyset$. Also by the assumption there is a proper subshift $Y_0\subseteq \phi[Y]$ such that $Y_0\cap \bigcup_{D'\in\phi'[\AAA]} Y_D=\emptyset$. It follows that \[W:=\bigcup_{D'\in\phi'[\AAA]} Y_D\] is a proper subshift of $\phi[Y]$ that intersects non-trivially every set $D$ corresponding to $D'\in\phi'[\AAA]$. Thus the corresponding $W'\subseteq \phi'[Y']$ is a proper subshift of $\phi[Y']$ that intersects non-trivially every set from $\phi'[\AAA]$, thus $\phi'[Y]$ is not $\phi'[\AAA]$-minimal, a contradiction.

So it remains to show that $Y$ is effective. Notice however that $Y'$ is sofic, thus effective, so we finish by the application of Lemma~\ref{lem:effectivesubshiifts}.
\end{proof}

\begin{corollary}\label{cor:nilpotentnoSTRP}
No infinite finitely generated nilpotent group that is not virtually cyclic has the strong topological Rokhlin property.
\end{corollary}
\begin{proof}
Let $H$ be a finitely generated nilpotent group that is not virtually cyclic. Recall that $H$ is then finitely presented (see e.g. \cite[Propositions 13.75 and 13.84]{DruKap-book}) and every subgroup of $H$ is finitely generated (see e.g. \cite[Theorem 13.57]{DruKap-book}). Moreover, $H$ has $\Int^2$ as a quotient since considering the upper central series $1=Z_1\trianglelefteq Z_2\trianglelefteq\ldots\trianglelefteq Z_n=H/\mathrm{Tor}(H)$, the quotient $(H/\mathrm{Tor}(H))/Z_{n-1}$ is a finitely generated torsion-free abelian group (see \cite[Lemma 13.69]{DruKap-book}) that cannot be cyclic. In particular, both $H$ and $\Int^2$ are finitely generated and recursively presented, and moreover the kernel of the quotient map from $H$ onto $\Int^2$ is finitely generated. So in order to apply Theorem~\ref{thm:disjointunionpropersubshifts} it is enough to show that $\Int^2$ satisfies the conditions imposed on $G$ in the statement of the theorem. This follows from the results in \cite{Hoch12}. The desired sofic subshift $X\subseteq A^{\Int^2}$ is the subshift denoted by $Z$ in \cite[Theorem 5.3]{Hoch12}. We shall use the following properties of $X$: (1) $X$ is a disjoint union of minimal subshifts (see \cite[Construction 5.2]{Hoch12}); (2) for every effective subshift $Y\subseteq X$, every minimal subshift of $Y$ is an accumulation point of minimal subshifts of $Y$ (see \cite[Lemma 5.5]{Hoch12}).

Now suppose that $Z$ is a subshift of finite type factoring on $X$ via some factor map $\phi$. To show that for every effective subshift $Y\subseteq Z$, $\phi[Y]\subseteq X$ is a disjoint union of infinitely many subshifts, notice first that $\phi[Y]$ is then itself effective. Since by (1) $X$ is a disjoint union of minimal subshifts, so is $\phi[Y]\subseteq X$. Since every minimal subshift of $\phi[Y]$ is an accumulation point of other minimal subshifts of $\phi[Y]$ by (2), this shows that this union must be infinite.  This finishes the proof.
\end{proof}

We now need few notations. First, if $H\leq G$ are a group and a subgroup, $A$ is a non-trivial finite set, and $X\subseteq A^G$ is a subshift, the \emph{$H$-projective subdynamics} of $X$ is the subshift \[\{x\upharpoonright H\colon x\in X\}\subseteq A^H.\]

Moreover, we say that an element $x\in A^G$ is \emph{Toeplitz} if for every $g\in G$ the set \[\{g^{-1}h\colon x(h)=x(g)\}\subseteq G\] is a finite-index subgroup. We refer the reader to \cite{Krie} for details on Toeplitz elements in subshifts over general countable groups.

Recall also that a group is \emph{indicable} if it admits an epimorphism onto $\Int$.\medskip

We shall also need the notion of \emph{(a non-trivial) Medvedev degree} (see \cite[Section 3.2]{Hoch12}). Let $X\subseteq 2^\Nat$ be a set. We say that it is \emph{effective} if its complement is a recursive set of cylinders (cf. with Defintion~\ref{def:effectiveshubshift}). We say that an effective set $X\subseteq 2^\Nat$ has \emph{non-trivial Medvedev degree} if there is no computable function computing an element of $X$. Let $X,Y\subseteq 2^\Nat$ be effective sets. If there is a computable function from $X$ into $Y$ and $Y$ has non-trivial Medvedev degree, then clearly so does $X$.

In the following result we use the simulation theorem of Barbieri from \cite{Bar19} where we `simulate' certain effective subshifts as projective subdynamics of sofic subshifts

\begin{theorem}
Let $G$ be a  finitely generated recursively presented indicable group. Let $H_1,H_2$ be finitely generated recursively presented groups. Then $G\times H_1\times H_2$ does not have the strong topological Rokhlin property. 
\end{theorem}
\begin{proof}
Fix $G$,  and $H_1$ and $H_2$ as in the statement. Fix also an epimorphism $f:G\twoheadrightarrow \Int$. We show that $G\times H_1\times H_1$ satisfies the requirement on the group $G$ from the statement of Theorem~\ref{thm:disjointunionpropersubshifts}. The conclusion will then follow by Theorem~\ref{thm:disjointunionpropersubshifts}.

Let $g\in G$ be an arbitrary element satisfying $f(g)=1\in\Int$. Notice that $G$ splits as a semi-direct product $\langle g\rangle \ltimes K$, where $K\subseteq G$ is the kernel of the epimorphism $f$ and $\langle g\rangle$ is infinite cyclic. We also fix a finite symmetric generating set $S$ of $G$ containing $g$. We now follow Hochman's argument from \cite[Section 5.2]{Hoch12} to code effectively closed subsets of $2^\Nat$ by effectively closed subshifts of $3^G$. We also correct a mistake that appeared in that coding in \cite[Section 5.2]{Hoch12} (notice that the map defined there cannot distinguish between the elements $10000\ldots$ and $01111\ldots$). Let $\omega\in 2^\Nat$, where again $2$ is here identified with the two-element set $\{1,2\}$. We define $z_\omega\in 3^G$ as follows. We choose an arithmetic progression $A_1\subseteq \Int$ of period $3$ passing through $0$ and we set $z_\omega(g^i)=\omega(1)$ for $i\in A_1$. We also set $z_\omega(g^i)=3$ for $i\in 1+A_1:=\{j+1\colon j\in A_1\}$. Next let $A_2$ be an arithmetic progression of period $9$ passing through $2$ and we set $z_\omega(g^i)=\omega(2)$ for $i\in A_2$ and $z_\omega(g^j)=3$ for $j\in 3+A_2$. Next we take an arithmetic progression $A_3$ of period $27$ passing through $8$ and set $z_\omega(g^i)=\omega(3)$ for $i\in A_3$, etc. Finally, we set $z_\omega(h)=z_\omega(g^i)$ if and only if $h=g^ik$, where $k\in K$.

Clearly, $z_\omega$ is a Toeplitz configuration, so \[Z_\omega:=\overline{G\cdot z_\omega}\subseteq 3^G\] is a minimal subshift by \cite[Corollary 3.4]{Krie}. Moreover, for every $z\in Z_\omega$ there is an algorithm that computes $\omega$ with $z$ as an input. Indeed, $z_\omega$ is invariant under the shift by the subgroup $K$, therefore so is $z$, and the algorithm only reads the data from $z\upharpoonright \{g^i\colon i\in\Int\}$. The element $\omega$ can then be recovered as in \cite[Section 3.1]{Bar19}.

Moreover, if $\Omega\subseteq 2^\Nat$ is an effectively closed subset, defining \[Z:=\bigcup_{\omega\in\Omega} Z_\omega,\] we claim that
\begin{itemize}
	\item the Medvedev degree of $Z$ is at least that of $\Omega$; in particular, if $\Omega$ has non-trivial Medvedev degree, then so does $Z$;
	\item $Z$ is an effectively closed subshift of $3^G$.
\end{itemize}
The argument from the paragraphs above show that $\Omega$ is computable from $Z$, so the Medvedev degree is at least the Medvedev degree of $\Omega$ and if $\Omega$ has non-trivial Medvedev degree, so does $Z$.

We clearly have that $Z$ is closed under the shift action of $G$, so we need to show it is effective, i.e. we need to show that there is an algorithm that given a finite pattern $p\in 3^F$, for some finite $F\subseteq G$, decides whether it is forbidden in $Z$. Given such a pattern $p\in 3^F$, the algorithm first checks whether $p$ is invariant under $f$, i.e. whether for $h\neq h'\in F$ such that $f(h)=f(h')$ we have $p(h)=p(h')$. This is verifiable by an algorithm and if $p$ is not $f$-invariant, then the algorithm considers $p$ to be forbidden. If on the other hand $p$ is $f$-invariant, the algorithm decides whether $p$ is a fragment of a coding of some $\omega\in 2^\Nat$ into $z_\omega$. If not, then $p$ is considered as forbidden. If yes, then the algorithm appeals to the algorithm computing that $\Omega$ is effectively closed to semi-decide whether $p$ is forbidden.\medskip

Now let $Z'\subseteq Z$ be an effectively closed subshift. We claim that it is a disjoint union of infinitely many proper subshifts. Since $Z$ is a disjoint union of minimal subshifts, so is $Z'$. Therefore we just need to show that $Z'$ is a union of infinitely many of them. This however follows from \cite[Lemma 5.4]{Hoch12} - notice it was proved there for $G=\Int^d$, but it works generally.

Now we apply \cite[Theorem 4.2]{Bar19} with $G$, $H_1$, $H_2$, and the effectively closed subshift $Z\subseteq 3^G$ to obtain a sofic subshift $Y\subseteq B^{G\times H_1\times H_2}$, for some non-trivial finite set $B$, such that the restriction of the shift action of $G$ to $Y$ is conjugate to $X$ and the $G$-projective subdynamics of $Y$ is equal to $Z$, and $H_1\times H_2$ acts trivially. Since $Y$ is sofic, let $X\subseteq D^{G\times H_1\times H_2}$, for some non-trivial finite set $D$, be a subshift of finite type that factors via some $\phi:X\twoheadrightarrow Y$ onto $Y$. In order to finish the proof, we need to show that for every effective subshift $X'\subseteq X$, $\phi[X']\subseteq Y$ is a disjoint union of infinitely many proper subshifts. Let $X'\subseteq X$ be effective. Then since symbolic factors of effective subshifts are effective, $\phi[X']\subseteq Y$ is effective as well. Since $G$-projective subdynamics is $Z$, subshifts of $Y$ are in $1-1$ correspondence with subshifts of $Z$, so $\phi[X']$ corresponds to an effective subshift of $Z$ which is by the argument above a disjoint union of infinitely many disjoint minimal subshifts.
\end{proof}

\subsection{STRP for groups that are not finitely generated}
In this section we investigate the strong topological Rokhlin property for groups that are not finitely generated. We have already mentioned that Kechris and Rosendal noticed in \cite{KeRo} that the free group on countably infnitely many generators does not have the strong topological Rokhlin property. It is plausible, although we cannot prove it at the moment, that finite generation is a necessary condition for having the strong topological Rokhlin property. One however cannot immediately dismiss the idea that a non-finitely generated group with the STRP could exist. Interesting recent observations of Barbieri \cite{Bar22} point out that the problem of existence of a strongly aperiodic subshift of finite type for groups that are not finitely generated also looks impossible on the first sight, however he shows that the `first sight' is wrong in this case.

Our most general result is the following.
\begin{theorem}\label{thm:nonFGgroups}
Let $G$ be a countable group. If for every finitely generated subgroup $H\leq G$ there exists $g\in G\setminus H$ such that one of the conditions below holds:
\begin{enumerate}
	\item $g$ centralizes $H$, i.e. $gh=hg$ for all $h\in H$;
	\item The subgroup $\langle g, H\rangle$ is equal to the free product $\langle g\rangle\ast H$;
\end{enumerate}
then $G$ does not have the strong topological Rokhlin property.
\end{theorem}
\begin{proof}
We shall show something much stronger. We show that for such $G$ and for any finite $A$ with at least two elements, the only projectively isolated subshifts in $\SH_G(A)$ are the singletons, i.e. the monochromatic configurations. An application of Proposition~\ref{prop:noRokhlin} then immediately gives that $G$ does not have the strong topological Rokhlin property.\medskip

Fix $G$ and $A$. Assume that $X\subseteq A^G$ is a projectively isolated subshift which is witnessed by some subshift of finite type $Y\subseteq B^G$, for some finite alphabet $B$, some open neighborhood $\NN$ of $Y$ and a factor map $\phi$ defined on every $Z\in\NN$ so that $\phi[Z]=X$. Applying Lemma~\ref{lem:standardformoffactormaps} and Proposition~\ref{prop:SFTprop} we may and will without loss of generality suppose that
\begin{itemize}
	\item $\phi$ is induced by a map $\phi_0:B\rightarrow A$ between the alphabets;
	\item $\NN$ is of the form $\NN_Y^F$ for some finite $F$ and $Y$ is of the form $X_{\VV_F}$ for some $F$-Rauzy graph $\VV_F=(V,(E_f)_{f\in F})$. In particular, $B=V$.
\end{itemize}

Let $H:=\langle F\rangle\leq G$ be the subgroup of $G$ generated by $F$. By our assumption, there exists $g\in G\setminus H$ such that either $g$ centralizes $H$ or has no relations with $H$, i.e. $\langle g, H\rangle$ is equal to the free product $\langle g\rangle\ast H$.\medskip

Let us at first assume the latter. We define a new $F'$-Rauzy graph $\VV'_{F'}:=(V, (E'_f)_{f\in F'})$, where
\begin{itemize}
	\item $F':=F\cup\{g\}$;
	\item $E'_f=E_f$, for $f\in F$;
	\item $E'_g$ is $\{(v,v)\colon v\in V\}$, the directed graph on $V$ consisting of all loops.
\end{itemize}

Set $Z:=X_{\VV'_{F'}}$. It is immediate that $Z_F\subseteq Y_F$. We need to check that $Z$ is non-empty, $Z_F=Y_F$ so that $Z\in\NN$, and that $\phi[Z]\neq X$.

Denote by $\bar H$ the subgroup $\langle g,H\rangle$. Let $p_1$ be the identity map on $H$ and let $p_2:\langle g\rangle:\rightarrow\{1_G\}$ be the map sending $g$ to the unit. Set \[p:=p_1\ast p_2: \bar H\rightarrow H.\]
Notice that every $h\in \bar H$ can be uniquely written as a word $h_1 h_2\ldots h_n$, where for $i<n$, $h_i\in H$ and $h_{i+1}\in\langle g\rangle$, or vice versa, and assuming for concreteness that e.g. $h_1\in H$ and $h_n\in \langle g\rangle$ we have in that case \[p(h_1 h_2\ldots h_n)=h_1 h_3\ldots h_{n-1}.\] 

To check that $Z$ is non-empty, pick any $y\in Y$. Let $(g_n)_{n\in\Nat}$ be an arbitrary set of left coset representatives for $\bar H$ in $G$. We define $z\in B^G$ as follows. Any $g\in G$ can be uniquely written as $g_n\cdot h$ for some $n\in\Nat$ and $h\in\bar H$ and we set \[z(g):=y(p(h)).\]
We claim that $z\in Y$ witnesses that $Z$ is non-empty since $z$ contains only allowed patterns as defined by $\VV'_{F'}$. This needs to be checked only on the left cosets of $\bar H$ since the forbidden patterns are determined by a finite subset of $\bar H$. We check it for the coset $\bar H$.

We need to check that for every $h\in\bar H$ and $f\in F'$, $(z(h),z(hf))$ is allowed. We do it by induction on the length of $h$ as a word $h_1 h_2\ldots h_n$, where for $i<n$, $h_i\in H$ and $h_{i+1}\in\langle g\rangle$, or vice versa. Suppose first that $n=1$. Either $h\in H$, or $h\in\langle g\rangle$. In the former case, if moreover $f\in F\subseteq H$, then $(z(h),z(hf))=(y(h),y(hf))$, and if $f=g$, then $z(h)=y(h)$ and $z(hf)=y(p(hf))=y(h)=z(h))$ which is again allowed. In the latter case, i.e. $h\in\langle g\rangle$ we have $z(h)=y(p(h))=z(1_G)$.  If $f\in H$ then arguing as before we have $(z(h),z(hf))=(y(1_G),y(f))$ which is allowed, or if $f=g$ we have $z(h)=z(hf)=y(1_G)$ which is again allowed. The general induction step is verified similarly and left to the reader. Moreover, by varying $y\in Y$ in the construction above we guarantee that $Z_F=Y_F$.\medskip

Now we check that $\phi[Z]\subsetneq X$. Notice first that since for every $z\in Z$ and $h\in G$ we have $z(h)=z(hg)$ and since $\phi$ is induced by $\phi_0:B\rightarrow A$ we get $\phi(z)(h)=\phi(z)(hg)$. Thus in order to show that $\phi[Z]\subsetneq X$ it is enough to find $x\in X$ and $h\in G$ such that $x(h)\neq x(hg)$.

Denote by $(v_n)_{n\in\Nat}$ some left coset representatives for $H$ in $G$, where $v_1=1_G$ and $v_2=g$. For $y\in B^G$ we have $y\in Y$ if and only if there exist $(z_n)_{n\in\Nat}\subseteq Y$ such that for any $n\in\Nat$ \[\forall h\in H\; \big(y(v_n h)=z_n(h)\big).\] Find arbitrarily some $(z_n)_{n\in\Nat}\subseteq Y$ such that \[\phi_0\big(z_1(1_G)\big)\neq \phi_0\big(z_2(1_G)\big),\] and define $y\in B^G$ so that for any $n\in\Nat$ \[\forall h\in H\; \big(y(v_n h)=z_n(h)\big).\] Then by above, we have $y\in Y$, however by definition \[\phi(y)(1_G)=\phi_0\big(z_1(1_G)\big)\neq \phi_0\big(z_2(1_G)\big)=\phi(y)(g).\] Since $\phi(y)\in X$, this is the desired contradiction.\bigskip

Now we assume that there exists $g\in G\setminus H$ that centralizes $H$. We distinguish two cases.
\begin{enumerate}
	\item Either $\bar H:=\langle g,H\rangle$ is equal to the direct product $H\times \langle g\rangle$,
	\item or $g$ is a root of a non-trivial element of $H$, i.e. there exist $h\in H\setminus \{1_G\}$ and $n\geq 2$ such that $g^n=h$ (where $h$ must be in the center of $H$).
\end{enumerate}

In the first case, we can define $Z$ to be $X_{\VV'_{F'}}$, where $\VV'_{F'}$ is exactly the same as above and we leave to the reader to verify that again $Z\in\NN$, $Z$ is non-empty, and $\phi[Z]\subsetneq X$.

So we now consider the case that there are $h\in H\setminus \{1_G\}$ and $n\geq 2$ such that $g^n=h$. We also assume that $n$ is the minimal $m\geq 2$ such that $g^m\in H$. We then set $Z\subseteq Y$ to be the subshift of finite type where for any $z\in B^G$, in order to be in $Z$, we require that (obviously) $z\in Y$ and there are no $h\in G$ and $1\leq i<j<k\leq n-1$ such that \[z(hg^i)\neq z(hg^j)\neq z(hg^k).\]
We check that $Z$ is non-empty and $Z_F=Y_F$, so that $Z\in\NN$. Clearly, $Z_F\subseteq Y_F$. Pick any $p\in Y_F$ and $y\in Y$ such that $y\upharpoonright F=p$. Let $(g_m)_{m\in\Nat}$ be some set of left coset representatives of $\bar H$ in $G$. We define $z\in B^G$ by setting for any $m\in\Nat$, $h\in H$ and $1\leq i\leq n-1$ \[z(g_m h g^i)=y(h).\] Clearly, $z\in Y$. Moreover, one can easily read off the definition that there are no $m\in\Nat$, $h\in H$ and $1\leq i<j<k\leq n-1$ such that $z(g_m h g^i)\neq z(g_m h g^j)\neq z(g_m h g^k)$, thus $z\in Z$, which shows both that $Z$ is non-empty and that $Z_F=Y_F$, so $Z\in\NN$.

We are left to check that $\phi[Z]\subsetneq X$. Since the defining window of $Y$ is $F$ and $g\notin\langle F\rangle$ one can construct an element $y\in Y$ such that $\phi_0(y(1_G))\neq \phi_0(y(g))\neq\phi_0(y(g^2))$, thus $X$ contains an element $x$ such that $x(1_G)\neq x(g)\neq x(g^2)$. By definition, there is no such an element in $\phi[Z]$. This finishes the proof.
\end{proof}

The following immediate corollary is another proof of the fact proved in \cite[2nd remark on page 331]{KeRo}.
\begin{corollary}
Let $G$ be the free group on countably many generators. Then $G$ does not have the strong topological Rokhlin property.
\end{corollary}
The next result also immediately follows from Theorem~\ref{thm:nonFGgroups}. It is perhaps not so surprising that the free abelian group on countably many generators does not have the strong topological Rokhlin property since already $\Int^d$, for $d\geq 2$, does not have it, however we get that even `one-dimensional' groups such as $\Rat$ do not have the strong topological Rokhlin property.
\begin{corollary}
Let $G$ be a group that contains a center that is not finitely generated. Then $G$ does not have the strong topological Rokhlin property. In particular, non-finitely generated abelian groups do not have the strong topological Rokhlin property.
\end{corollary}
One of the possible candidates for a non-finitely generated group that has the strong topological Rokhlin property could be the Hall universal locally finite group. Recall that the Hall group is defined as the unique countably infinite locally finite group which contains every finite group as a subgroup and where any two finitely generated subgroups are conjugated (see \cite{Hall}).
\begin{corollary}
The Hall universal locally finite group does not have the strong topological Rokhlin property.
\end{corollary}
\begin{proof}
We apply Theorem~\ref{thm:nonFGgroups}. Denote by $G$ the Hall group and let $H\leq G$ be some finitely generated subgroup. We verify that condition (2) of Theorem~\ref{thm:nonFGgroups} is satisfied. Since $G$ contains every finite group as a subgroup the direct product $H\times \Int_2$ embeds via some monomorphism $\phi$ into $G$. Denote by $f\in\Int_2$ the non-trivial element of order $2$. Set $H':=\phi[H\times\{1\}]$. Since $H'\leq G$ is isomorphic to $H$ there exists $g\in G$ such that $gH'g^{-1}=H$. It follows that $g\phi\big(1,f)\big)g^{-1}$ is a non-trivial element commuting with $H$. This finishes the proof.
\end{proof}

\section{Shadowing}\label{sect:genericdynamics}
This section contains the proof of the equivalence between \eqref{it:intro1-2} and \eqref{it:intro1-3} from Theorem~\ref{thm:intro1}.

As mentioned in Introduction, shadowing has been originally defined for actions of $\Int$, however now the notion is available for any countable group (see \cite{OsTi} and \cite{ChKeo}).
\begin{definition}
Let $G$ be a countable group acting on a compact metrizable space $X$. Let $d$ be a compatible metric on $X$. For $\delta>0$ and finite set $S\subseteq G$, a $G$-indexed set $(x_g)_{g\in G}\subseteq X$ is called a \emph{$(\delta,S)$-pseudo-orbit} if for every $g\in G$ and $s\in S$ we have \[d(s\cdot x_g,x_{sg})<\delta.\]

We say that the action has the \emph{shadowing}, or \emph{the pseudo-orbit tracing property}, if for any $\varepsilon>0$ there are $\delta>0$ and finite set $S\subseteq G$ such that for any $(\delta,S)$-pseudo-orbit $(x_g)_{g\in G}\subseteq X$ there exists $x\in X$ whose orbit $\varepsilon$-traces the pseudo-orbit, i.e. for every $g\in G$ \[d(x_g,g\cdot x)<\varepsilon.\]
\end{definition}
It is straightforward to verify that shadowing does not depend on the choice of the compatible metric and moreover that if $G$ is finitely generated then the finite set $S$ from the definition can be always taken to be some fixed finite generating set of $G$.

The following definition has its origin in \cite{GoMe} where it turned out to be crucial for describing inverse limits with shadowing for actions of $\Int$. The version for general countable groups appeared in \cite{LiChZh}.

\begin{definition}
Let $G$ be a countable group. Let $(X_n)_n$ be an inverse system of subshifts over $G$, where the bonding maps $(\phi_n^m: X_n\rightarrow X_m)_{n\geq m\in\Nat}$ are not necessarily onto, however for every $n_0\in\Nat$ there exists $n\in\Nat$ such that for every $m\geq n$ we have $\phi_m^{n_0}[X_m]=\phi_n^{n_0}[X_n]$. Then we say that the inverse system satisfies the \emph{Mittag-Leffler condition}.
\end{definition}

The main results of \cite{GoMe} and \cite{LiChZh} characterize actions with the shadowing property of some fixed group $G$ on the Cantor space as precisely those that are conjugate to inverse limits of inverse systems of subshifts of finite type over $G$ satisfying the Mittag-Leffler property. They will be applied in the following theorem.

\begin{theorem}\label{thm:genericshadowing}
Let $G$ be a finitely generated group with the strong topological Rokhlin property. Then the generic action of $G$ on the Cantor space has shadowing. In particular, shadowing is generic in $\Act_G(\CC)$.
\end{theorem}

Before proceding to the proof we shall need the following proposition.
\begin{proposition}\label{prop:inverselimitofprojiso}
Let $G$ be a countable group admitting a generic action $\alpha\in\Act_G(\CC)$. Then $\alpha$ is an inverse limit of projectively isolated subshifts with isolated factor maps as the bonding maps.
\end{proposition}
\begin{proof}
Let $G$ and $\alpha$ be as in the statement. By Theorem~\ref{thm:mainRokhlin}, projectively isolated subshifts are dense in $\SH_G(n)$, for all $n\geq 2$. Denote now by $\mathbb{P}$ the countable set of all clopen partitions of $\CC$ and set \[\begin{split}\GG:=\big\{& \beta\in\Act_G(\CC)\colon \forall \PP\in\mathbb{P}\;\exists \PP'\in\mathbb{P}\;\exists X\in\SH_G(\PP') \exists F\subseteq_\mathrm{fin} G\\ & \big(\PP'\preceq \PP\wedge \forall Z,Z'\in\NN_X^F(\phi_{\PP'}^\PP[Z]=\phi_{\PP'}^\PP[Z'])\wedge\QQ(\beta,\PP')\in\NN_X^F\big) \big\}.\end{split}\] We claim that $\GG$ is a dense $G_\delta$ set consisting of (all the) actions from $\Act_G(\CC)$ that are inverse limits of projectively isolated subshifts. To check that $\GG$ is $G_\delta$, it suffices to show that for fixed partitions $\PP$ and $\PP'$ such that $\PP'\preceq\PP$, for fixed $X\in \SH_G(\PP')$ and its neighborhood $\NN_X^F$ with the property that $\phi_{\PP'}^\PP[Z]=\phi_{\PP'}^\PP[Z']$, for all $Z,Z'\in\NN_X^F$ (notice that this last condition does not depend on $\beta$), the set $\{\beta\in\Act_G(\CC)\colon \QQ(\beta,\PP')\in \NN_X^F\}$ is open. This however follows since the map $\QQ(\cdot,\PP')$ is continuous by Proposition~\ref{prop:Qcontinuity}.

Next we show that every $\beta\in\GG$ is an inverse limit of projectively isolated subshifts. In fact, the converse is true as well, i.e. every $\beta\in\Act_G(\CC)$ that is an inverse limit of projectively isolated subshifts belongs to $\GG$. Since we do not need this, it is left to the reader. Fix $\beta\in\GG$ and some compatible metric $d$ on $\CC$, and let $\PP_1$ be an arbitrary clopen partition of $\CC$ whose all elements have diameter less than $1/2$ with respect to $d$. Since $\beta\in\GG_1$ there exists a refinement $\PP_2\preceq\PP_1$ such that $\QQ(\beta,\PP_2)$ and the map $\phi_{\PP_2}^{\PP_1}: \QQ(\beta,\PP_2)\rightarrow \QQ(\beta,\PP_1)$ between the subshifts defined by the inclusion map $\PP_2\to\PP_1$ witnesses that $\QQ(\beta,\PP_1)$ is projectively isolated. By Lemma~\ref{lem:projshiftproperties}, refining $\PP_2$ if necessary, without loss of generality we may assume that all elements of $\PP_2$ have diameter less than $1/2^2$.

We repeat the argument with $\PP_2$ to obtain $\PP_3\preceq \PP_2$, whose elements we may assume have diameter less than $1/2^3$, such that $\QQ(\beta,\PP_3)$ and the map $\phi_{\PP_3}^{\PP_2}: \QQ(\beta,\PP_3)\rightarrow \QQ(\beta,\PP_2)$, defined again by the inclusion $\PP_3\to\PP_2$, witnesses that $\QQ(\beta,\PP_2)$ is projectively isolated. We continue analogously to obtain partitions $\PP_n\preceq\PP_{n-1}$, whose elements have diameter less than $1/2^n$, for all $n\in\Nat$.

We claim that the inverse limit of \[\QQ(\beta,\PP_1)\xleftarrow{\phi_{\PP_2}^{\PP_1}}\QQ(\beta,\PP_2)\xleftarrow{\phi_{\PP_3}^{\PP_2}}\QQ(\beta,\PP_3)\xleftarrow{\phi_{\PP_4}^{\PP_3}}\ldots\] is equal to $\beta$. Indeed, it is clear that the map \[x\in\CC\mapsto \big(Q^\beta_{\PP_n}(x)\big)_{n\in\Nat}\] is a factor map onto the inverse limit. So it suffices to check that it is one-to-one. Pick $x\neq y\in\CC$. Since the diameters of the elements of the partitions $(\PP_n)_{n\in\Nat}$ tend to $0$, there exists $n\in\Nat$ such that $x$ and $y$ lie in different elements of the partition $\PP_n$, so $Q^\beta_{\PP_n}(x)\neq Q^\beta_{\PP_n}(y)$, and therefore the map above is one-to-one.\medskip

Notice that in the previous paragraphs we did not use that projectively isolated subshifts are dense which we will do now when proving that $\GG$ is dense. Let $\UU\subseteq\Act_G(\CC)$ be open and we may suppose it is of the form $\NN_\gamma^{F,\PP}$ for some $\gamma\in\Act_G(\CC)$, finite symmetric $F\subseteq G$ containing $1_G$, and a clopen partition $\PP$. Set $X:=\QQ(\gamma,\PP)$ and recall that by Proposition~\ref{prop:Qmap-nbhds}\eqref{it1-Qmap-nbhds} we have $\QQ(\cdot,\PP)[\NN_\gamma^{F,\PP}]=\NN_X^F$. Since projectively isolated subshifts are dense there is a projectively isolated subshift $X_1\in\NN_X^F$ which is witnessed by some isolated factor map $\phi_2^1:Y\to X$, for some subshift $Y\subseteq B^G$, and some neighborhood $\NN_Y^{F_2}$ such that for all $Y'\in\NN_Y^{F_2}$ we have $\phi_2^1[Y']=X_1$. Pick a projectively isolated subshift $X_2\in\NN_Y^{F_2}$. Continuing analogously, we obtain an inverse sequence $X_1\xleftarrow{\phi_2^1} X_2\xleftarrow{\phi_3^2}X_3\xleftarrow{\phi_4^3}\ldots$ of projectively isolated subshifts whose inverse limit $\lambda$ is, without loss of generality by taking a product with a trivial action on $\CC$, an action on $\CC$. Since by Lemma~\ref{lem:SFTfactorofSFT} and Corollary~\ref{cor:projisolatedconjugacy}, we may assume that the maps $\phi_n^{n-1}$ are induced by maps between the corresponding alphabets $A_n$, resp. $A_{n-1}$, for each $n\in\Nat$ the partition $\Big\{x\in X_n\colon x(1_G)=a\}\colon a\in A_n\Big\}$ induces a partition $\PP_n$ of $\CC$ such that $\PP_{n+1}\preceq\PP_n$, for each $n\in\Nat$, and $\lim_{n\to\infty} \max\{\mathrm{diam}_d(P)\colon P\in\PP_n\}=0$. It is plain to check that $\lambda\in\GG$ then. It is also straightforward to check that a conjugate of $\lambda$ is in $\NN_\gamma^{F,\PP}$ by the same arguments as in the proof of Proposition~\ref{prop:Qmap-nbhds}. Since $\GG$ is clearly conjugacy invariant, it follows that this conjugate is in $\GG$, so since $\UU$ was arbitrary, $\GG$ is dense.

The proof is finished by noticing that the conjugacy class of $\alpha$ must intersect the conjugacy invariant set $\GG$, since both are dense $G_\delta$, thus $\alpha\in\GG$ and we are done.
\end{proof}
\begin{proof}[Proof of Theorem~\ref{thm:genericshadowing}]
Let $G$ be as in the statement and let $\alpha\in\Act_G(\CC)$ be the generic action. By Proposition~\ref{prop:inverselimitofprojiso}, $\alpha$ is an inverse limit of some sequence $(X_n)_{n\in\Nat}$ of projectively isolated subshifts $X_n\subseteq A_n^G$ with respect to isolated factor maps $(\phi_n^m:X_n\to X_m)_{n\geq m\in\Nat}$.\medskip

We define a sequence $(Y_n)_{n\in\Nat}$ of subshifts of finite type. Let $Y_1$ be any subshift of $A_1^G$ that is of finite type and such that $X_1\subseteq Y_1$. Pick now $n\geq 2$. By definition, the isolated factor map $\phi_n^{n-1}: X_n\rightarrow X_{n-1}$ is defined on some neighborhood $\NN_{X_n}^{F_n}$, where $F_n\subseteq G$ is a finite set, such that for every $Y\in\NN_{X_n}^{F_n}$, we have $\phi_n^{n-1}[Y]=X_{n-1}$. Notice that by Lemma~\ref{lem:SFTnbhrds} and the assumption, $\NN_{X_n}^{F_n}$ contains a subshift of finite type $Y_n$ such that $X_n\subseteq Y_n$ and $\phi_n^{n-1}[Y_n]=X_{n-1}$.

We claim that the system $(Y_n)_{n\in\Nat}$ with bonding maps $(\phi_n^m)_{n\geq m\in\Nat}$ is an inverse system satisfying the Mittag-Leffler condition and the inverse limit is equal to $\alpha$. The former is straightforward, we show that the inverse limits of the sequences $(X_n)_n$, resp. $(Y_n)_n$ are equal. Pick $(x_n)_n\in \underset{n\to\infty}{\varprojlim} X_n$. Since for every $n\in\Nat$, $X_n\subseteq Y_n$, clearly $(x_n)_n\in \underset{n\to\infty}{\varprojlim} Y_n$. Conversely, pick $(y_n)_n\in \underset{n\to\infty}{\varprojlim} Y_n$. We claim that for every $n\in\Nat$, $y_n\in X_n$, and thus $(y_n)_n\in \underset{n\to\infty}{\varprojlim} X_n$. Indeed, for any $n\in\Nat$ we have $\phi_{n+1}^n(y_{n+1})=y_n$. Since $y_{n+1}\in Y_{n+1}$ and $\phi_{n+1}^n[Y_{n+1}]=X_n$, we get $y_n\in X_n$.\medskip

Finally, we apply \cite[Theorem 1.2]{LiChZh} with $(Y_n)_n$ to get that $\alpha$ has shadowing.
\end{proof}

As a corollary of Theorems~\ref{thm:genericshadowing} and~\ref{thm:STRPfreeproducts} we get the following.
\begin{corollary}
Let $G=\bigstar_{i\leq n} G_i$, where for $i\leq n$, $G_i$ is finite or cyclic. Then shadowing is generic in $\Act_G(\CC)$.
\end{corollary}
The next step is to prove the converse, i.e. actions with shadowing form a meager set in $\Act_G(\CC)$ for groups $G$ without the strong topological Rokhlin property. First we reformulate the original definition of pseudo-orbit tracing property to an equivalent definition for zero-dimensional spaces that is more convenient for us as it employs clopen partitions instead of compatible metrics.
\begin{lemma}\label{lem:shadowinonCantor}
Let $G$ be a countable group acting continuously on a zero-dimensional compact metrizable space $X$. Then the action has shadowing if and only if for every clopen partition $\PP$ of $X$ there exist a refinement $\PP'\preceq \PP$ and a finite set $S\subseteq G$ so that for every $(x_g)_{g\in G}\subseteq X$ satisfying that for all $g\in G$ and $s\in S$ we have $x_{sg}$ and $s x_g$ lie in the same element of the partition $\PP'$, there is $x\in X$ such that $gx$ and $x_g$ lie in the same element of $\PP$ for every $g\in G$.
\end{lemma}
\begin{proof}
The straightforward proof is similar to the proof of Lemma~\ref{lem:basicopennbhds} and left to the reader.
\end{proof}

By definition, every symbolic factor of a subshift of finite type is sofic. Since subshifts of finite type are precisely those subshifts with shadowing (this is easy to check, a formal proof can be found in \cite{ChKeo}), the following proposition is a generalization of this fact - that a factor of an SFT is sofic.
\begin{proposition}\label{prop:shadowingfactorsonsofic}
Let $G$ be a countable group acting continuously on a zero-dimensional compact metrizable space $X$ so that the action has shadowing. Then every symbolic factor of the action is sofic.
\end{proposition}
\begin{proof}
Fix the group $G$ and the zero-dimensional compact metrizable space $X$ on which $G$ acts continuously with the shadowing property. Denote the action by $\alpha$. Let $\phi:\alpha\rightarrow Y$ be a factor map, where $Y\subseteq A^G$ is a subshift for some non-trivial finite set $A$. We show that $Y$ is sofic. By Lemma~\ref{lem:factoringonshift} there exists a clopen partition $\PP$ of $X$, with $|\PP|=|A|$, so that $\phi=Q_\PP^\alpha$. By the assumption and Lemma~\ref{lem:shadowinonCantor}, there exist a refinement $\PP'\preceq\PP$  and a finite set $S\subseteq G$, which we may assume to be symmetric and containing $1_G$, so that for every $(x_g)_{g\in G}\subseteq X$ satisfying that for all $g\in G$ and $s\in S$ we have $x_{sg}$ and $s x_g$ lie in the same element of the partition $\PP'$, there is $x\in X$ such that $gx$ and $x_g$ lie in the same element of $\PP$ for every $g\in G$.

We define a subshift $Z\subseteq (\PP')^G$ of finite type with defining window $S$ as follows. A pattern $p\in (\PP')^S$ is allowed if and only if there exists $x\in X$ such that for every $s\in S$ and $P\in\PP'$ \[p(s)=P\Leftrightarrow s^{-1}x\in P.\] Let $f_0:\PP'\rightarrow \PP$ be the inclusion map, i.e. $P\subseteq f_0(P)$ for every $P\in\PP'$. It induces a continuous $G$-equivariant map $\psi:Z\rightarrow \PP^G$. We claim that $\psi[Z]=Y$. The proof of the claim will finish the proof of the proposition since $Y$ will then be a factor of a subshift of finite type, therefore sofic.\medskip

Let us first show that $Y\subseteq \psi[Z]$. This is a general argument where the shadowing property is not yet used. Set $\phi':=Q_{\PP'}^\alpha$ and notice that
\begin{itemize}
	\item $\phi=\psi\circ\phi'$,
	\item $\phi'[X]\subseteq Z\subseteq (\PP')^G$ since by the definition of $Z$, we have $(\phi[X])_S=Z_S$.
\end{itemize}

Pick $y\in Y$ and $x\in X$ such that $\phi(x)=y$. For $z:=\phi'(x)\in \phi'[X]\subseteq Z$ we then have \[\psi(z)=\psi\circ\phi'(x)=\phi(x)=y.\]

Finally we show that $\psi[Z]\subseteq Y$. Pick $z\in Z$ and let us show that $\psi(z)\in Y$. For every $g\in G$, $p_g:=gz\upharpoonright S\in (\PP')^S$ is an allowed pattern of $Z$, so by the definition of $p_g$ there exists $x_g\in X$ such that for every $s\in S$ and $P\in\PP'$ \[p(s)=P\Leftrightarrow s^{-1}x_g\in P.\] It follows that $(x_g)_{g\in G}\subseteq X$ has the property that for every $g\in G$ and $s\in S$ we have $x_{sg}$ and $sx_g$ lie in the same element of the partition of $\PP'$. Then, applying the pseudo-orbit tracing property with respect to $\PP'$ and $\PP$, there exists $x\in X$ such that for every $g\in G$ we have that $gx$ and $x_g$ lie in the same element of $\PP$. We claim that $\phi(x)=\psi(z)$. Pick $g\in G$ and let us show that $\phi(x)(g)=\psi(z)(g)$. For $P\in\PP$ we have \[\begin{split}& \phi(x)(g)=P\Leftrightarrow g^{-1}x\in P\Leftrightarrow x_{g^{-1}}\in P\Leftrightarrow p_{g^{-1}}(1_G)\subseteq P\Leftrightarrow\\ & g^{-1}z(1_G)\subseteq P\Leftrightarrow z(g)\subseteq P\Leftrightarrow \psi(z)(g)=P.\end{split}\]
\end{proof}

\begin{proposition}\label{prop:NoSTRPnonsoficfactors}
Let $G$ be a countable group that does not have the strong topological Rokhlin property. Then the set \[\AAA:=\{\alpha\in\Act_G(\CC)\colon \exists \PP\text{ clopen partition }\;(\QQ(\alpha,\PP)\text{ is not sofic})\}\] is non-meager.
\end{proposition}
\begin{proof}
Fix $G$ without the strong topological Rokhlin property. By Proposition~\ref{prop:noRokhlin}, there exist a non-trivial finite set $A$ and an open set $\NN\subseteq\SH_G(A)$ that contains no projectively isolated subshift and must be therefore infinite as isolated points are projectively isolated. Applying Lemma~\ref{lem:SFTnbhrds}, without loss of generality, we may assume that $\NN=\NN_X^F$, where $X\subseteq A^G$ is a subshift of finite type, $F\subseteq G$ is a finite subset, and for every $X'\in\NN$ we have $X'\subseteq X$. By Proposition~\ref{prop:Qcontinuity}, there exist $\alpha\in\Act_G(\CC)$ and a clopen partition $\PP$ of $\CC$ such that $\QQ(\alpha,\PP)=X$. Moreover, again by Proposition~\ref{prop:Qcontinuity}, by the continuity of $\QQ(\cdot,\PP)$ there exists an open neighborhood of $\alpha$, which we may suppose to be of the form $\NN_\alpha^{E,\PP'}$, for some finite set $E\subseteq G$ and a clopen partition $\PP'\preceq \PP$ such that $\QQ(\cdot,\PP)[\NN_\alpha^{E,\PP'}]\subseteq \NN_X^F$. In particular, for every $\beta\in\NN_\alpha^{E,\PP'}$ we have $\QQ(\beta,\PP)\subseteq X$.

Let $\{X_n\colon n\in\Nat\}$ be an enumeration of all sofic subshifts inside $\NN$. Notice that the set is indeed infinite since otherwise, as $\NN$ does not contain isolated points, $\NN$ would contain an open subset without any sofic subshift which is a contradiction as subshifts of finite type are dense - recall Lemma~\ref{lem:SFTnbhrds}. Let $m=|A|$ and for each $n\in\Nat$ set \[\begin{split}A_n:= & \{\alpha\in\Act_G(\CC)\colon\\ & \text{for no clopen partition }\RR=\{R_1,\ldots,R_m\}\text{ of }\CC,\; \QQ(\alpha,\RR)=X_n\}.\end{split}\]

By Lemma~\ref{lem:claimlemma}, $A_n$ is dense $G_\delta$. Suppose now that $\AAA$ is meager, so $\Act_G(\CC)\setminus\AAA$ is comeager, and we reach a contradiction. There exists an action \[\gamma\in \NN_\alpha^{E,\PP'}\cap \big(\Act_G\setminus \AAA\big)\cap \bigcap_{n\in\Nat} A_n,\] since the set on the right-hand side is non-meager, so non-empty. Since $\gamma\in\NN_\alpha^{E,\PP'}$ we get that $\QQ(\gamma,\PP)\subseteq X$. Since $\gamma\notin\AAA$ we get that $Y:=\QQ(\gamma,\PP)\subseteq X$ is sofic. However, then $Y=X_n$ for some $n\in\Nat$. Since $\gamma\in A_n$ it follows that $\QQ(\gamma,\PP)\neq X_n$, a contradiction.
\end{proof}

\begin{theorem}
Let $G$ be a countable group that does not have the strong topological Rokhlin property. Then the set \[\SSS:=\{\alpha\in\Act_G(\CC)\colon \alpha\text{ has shadowing}\}\] is dense, but meager. In particular, shadowing is not generic for actions of $G$.
\end{theorem}
\begin{proof}
Fix $G$ as in the statement. Density of $\SSS$ follows from the fact that subshifts of finite type have shadowing (see e.g. \cite{ChKeo}) and actions $\alpha\in\Act_G(\CC)$ conjugate to subshifts of finite type are dense by Proposition~\ref{prop:Qcontinuity} and Lemma~\ref{lem:SFTnbhrds}. To get meagerness, applying Fact~\ref{fact:0-1-law}, it suffices to show that $\SSS$ is not comeager. To reach a contradiction, suppose that $\SSS$ is comeager. Then it has a non-empty intersection with the set $\AAA$ from the statement of Proposition~\ref{prop:NoSTRPnonsoficfactors}, so there is $\alpha\in \SSS\cap\AAA$. It follows that there is a clopen partition $\PP$ of $\CC$ such that $\QQ(\alpha,\PP)$ is not sofic. That is however in contradiction with Proposition~\ref{prop:shadowingfactorsonsofic}.
\end{proof}
The following is an immediate corollary (cf. with e.g. \cite{Ko07}).
\begin{corollary}
Generically, Cantor space actions of $\Int^d$, for $d\geq 2$, or more generally of finitely generated nilpotent groups that are not virtually cyclic, do not have shadowing.
\end{corollary}

\section{Remarks, problems, and questions}\label{sect:problems}
It is our hope that this paper will stimulate further research in this area, even among researchers working in symbolic dynamics over general groups, and more applications of Theorem~\ref{thm:mainRokhlin} will be found. The following is the most general problem that we state and believe it is worth of attention.
\begin{problem}
Apply Theorem~\ref{thm:mainRokhlin} to a wider class of groups. That is, find more groups for which (strongly) projectively isolated subshifts are dense in the spaces of subshifts.
\end{problem}
We know very little about the permanence properties of the class of countable groups satisfying the STRP. We even do not know whether every virtually cyclic group satisfies the STRP. Inspired by the results from \cite{Cohen} we ask the following.
\begin{question}
Is the class of countable groups satisfying the STRP closed under commensurability? Under virtual isomorphism? Under quasi-isometry?
\end{question}
The following is based on the fact we do not know any projectively isolated subshift that is not strongly projectively isolated.
\begin{question}
Do there exist a group $G$ and a subshift $X\subseteq A^G$, for some non-trivial finite $A$, such that $X$ is projectively isolated, however not strongly projectively isolated?
\end{question}
We conclude with a question related to the results of Section~\ref{sect:noSTRP}.
\begin{question}
Does there exist a countable group that is not finitely generated, yet it still has the STRP?
\end{question}
\bigskip

\noindent{\bf Acknowledgements.} We would like to thank to Sebasti\' an Barbieri for explaining us the different notions of effective subshifts over general countable groups and to Alexander Kechris and Steve Alpern for several comments. We are also grateful to the anonymous referee for many useful comments, especially the suggestion to use Proposition~\ref{prop:BYMeTs} in the proof of Theorem~\ref{thm:mainRokhlin}.
\bibliographystyle{siam}
\bibliography{references-Rokhlin}
\end{document}